\documentclass[11pt,reqno]{amsart}
\usepackage{amsmath}
\usepackage{amssymb}
\usepackage{amstext}
\usepackage{amsfonts}
\usepackage{amssymb}
\usepackage{wasysym}

\usepackage[tikz]{bclogo}
\usepackage{graphicx}
\usepackage{floatrow}
\usepackage{times,graphicx,epsfig}
\usepackage{todonotes}

\usepackage{color}
\usepackage[latin1]{inputenc}
\usepackage{amsmath,amsfonts,amsthm,amssymb}

\theoremstyle{plain}
\newtheorem{theorem}{Theorem}[section]
\newtheorem{lemma}[theorem]{Lemma}
\newtheorem{corollary}[theorem]{Corollary}
\newtheorem{proposition}[theorem]{Proposition}

\newtheorem{definition}[theorem]{Definition}
\newtheorem{remark}[theorem]{Remark}
\theoremstyle{plain}

\def \R {{\mathbb {R}}}

\newcommand{\average}{{\mathchoice {\kern1ex\vcenter{\hrule height.4pt
width 6pt
depth0pt} \kern-9.7pt} {\kern1ex\vcenter{\hrule height.4pt width 4.3pt
depth0pt}
\kern-7pt} {} {} }}

\newcommand{\he}[1]{{\mathbb H}^{#1}}

\newcommand{\scalp}[3]{\langle {#1} , {#2}\rangle_{#3}}

\def \xn {{\tilde x}}   
\def \yn {{\tilde y}}   
\def \origin {{n}}      
\def \dimM {{Q }}     
\def \BalM {{B_M }}    
\def \tx {x_{5}}
\def \ty {y_{5}}

\def \R {{\mathbb {R}}}

\def  \e {{\varepsilon}}

\def \tilde {\widetilde}

\def \R {{\mathbb {R}}}
\def \H {{\mathbb {H}}}

\def \e {{\varepsilon}}

\def \phi {{\varphi}}

\def \tilde {\widetilde}

\definecolor{mypink1}{rgb}{0.858, 0.188, 0.478}
\definecolor{babypink}{rgb}{0.96, 0.76, 0.76}
\definecolor{amaranth}{rgb}{0.9, 0.17, 0.31}
\definecolor{amber(sae/ece)}{rgb}{1.0, 0.49, 0.0}
\definecolor{amber}{rgb}{1.0, 0.75, 0.0}
\definecolor{atomictangerine}{rgb}{1.0, 0.6, 0.4}
\definecolor{aureolin}{rgb}{0.99, 0.93, 0.0}
\definecolor{bittersweet}{rgb}{1.0, 0.44, 0.37}
\definecolor{blush}{rgb}{0.87, 0.36, 0.51}
\definecolor{bostonuniversityred}{rgb}{0.8, 0.0, 0.0}
\definecolor{brightpink}{rgb}{1.0, 0.0, 0.5}
\definecolor{brilliantrose}{rgb}{1.0, 0.33, 0.64}
\definecolor{brilliantlavender}{rgb}{0.96, 0.73, 1.0}
\definecolor{candyapplered}{rgb}{1.0, 0.03, 0.0}
\definecolor{ao}{rgb}{0.0, 0.5, 0.0}
\definecolor{britishracinggreen}{rgb}{0.0, 0.26, 0.15}
\definecolor{cadmiumgreen}{rgb}{0.0, 0.42, 0.24}
\definecolor{darkspringgreen}{rgb}{0.09, 0.45, 0.27}
\definecolor{dartmouthgreen}{rgb}{0.05, 0.5, 0.06}

\numberwithin{equation}{section} \makeatletter

\title[A representation formula on the characteristic plane]{
A representation formula for regular  functions \\ on the characteristic plane of the second Heisenberg group}
\author[A. Baldi, G. Citti, G. Cupini]{Annalisa Baldi, Giovanna Citti, Giovanni Cupini}
\address{Dipartimento di Matematica, Piazza di Porta S. Donato 5,
40126 Bologna, Italy}\email{annalisa.baldi2@unibo.it, giovanna.citti@unibo.it, giovanni.cupini@unibo.it}

\thanks{ 
The authors are supported by University of Bologna, funds for selected research 
topics, and by PRIN22 2022F4F2LH.  
The first and the third author are supported by GNAMPA of INdAM, Italy. The second author is supported by the PE12 MNESYS project.
}

\begin{document}

\begin{abstract} 
The aim of this paper is to study  a Laplace-type operator and its fundamental solution on  the characteristic plane in the Heisenberg group $\he 2$. We introduce a conformal version of the Laplacian and we prove that the distance induced  by the immersion  in the ambient space is a good approximation of its fundamental solution. We provide in particular a representation formula for smooth functions in terms of the gradient of the function and the gradient of the approximated fundamental solution. This representation formula in the plane is stable up to its characteristic point.
\end{abstract}

\maketitle
\tableofcontents

\section{Introduction}
Purpose of this paper is to study the Laplacian operator and its fundamental solution on  a characteristic plane in the Heisenberg group $\he 2$.

We will denote by
$ \tilde x =(x, x_{5})$ a general point of $\R^4 \times \R$, where $x\in \R^4$, $x_5\in \R$, and we will make the canonical choice of the vector 
fields 
\begin{equation}
\label{vf}
X_1 = \partial_{x_1} - \frac{x_3}{2} \partial_{x_5}, \quad X_2 = \partial_{x_2} - \frac{x_4}{2} \partial_{x_5}, 
\quad X_3 = \partial_{x_3} + \frac{x_1}{2} \partial_{x_5}, \quad X_4 = \partial_{x_4} + \frac{x_2}{2} \partial_{x_5}. 
\end{equation}
The space $\R^4 \times \R$,  with this choice of vector fields, 
can be identified with the Heisenberg group $\he 2$. 
The space generated by $X_1, X_2, X_3, X_4$ at any fixed point $\tilde x$ is called the {\it horizontal space}, denoted by $H_{\tilde x}\he 2$, and the metric  which makes these vectors orthogonal, 
induces a sub-Riemannian structure on $\he 2$. 
If $M$ is a regular manifold in  $\he 2$, its   horizontal tangent plane  at a point $\tilde x$ is defined as the intersection between the tangent plane $T_{\tilde x}M$ to the manifold $M$ at the point ${\tilde x}$  with the horizontal plane $H_{{\tilde x}}\he 2$, and it is denoted  
$H_{\tilde x}M$. 
In general $T_{\tilde x}M$ and $H_{\tilde x}M$ are distinct: if $M$ is a hyper-surface, the orthogonal complement of  
$H_{\tilde x}M$ in $T_{\tilde x}M$ is mono-dimensional, and its unitary vector  is called   the horizontal normal $\nu$ to $M$ at the point $\tilde x$.  
Points ${\tilde x}$ of $M$ at which  $T_{\tilde x}M$ and $H_{\tilde x}M$ coincide are called {\it characteristic points}, and it is known that  the set of characteristic points has measure $0$ (see \cite{D}, \cite{FW97} , \cite{DGN98}, \cite{B}).

The interest in studying calculus on manifolds in the Heisenberg group comes from two classical problems in subriemannian PDEs: boundary values problems and curvature problems. Indeed classes of regular functions  on hypersurfaces  in the Heisenberg groups were first introduced by Folland \cite{Folland}, Folland and Stein \cite{FollandStein} 
and Jerison 
\cite{Jerison}, 
\cite{Jerison2}, in order to study the  boundary value problem for the Kohn Laplacian. Later on,  calculus  on manifolds was further developed in connection with the notion of curvature, starting from the paper   \cite{DGN07}.

A systematic study of properties of manifolds without characteristic points, called intrinsic graphs has been performed by 
Franchi, Serapioni and 
Serra Cassano
(see for example 
\cite{FSSC96}, \cite{FSSC}, \cite{FSSC1}). 
The   regularity  on these manifolds can be expressed in terms of suitable non linear intrinsic vector fields  (see \cite{ASV}  and \cite{cittiman06}). Properties of intrinsic graphs,  
 area formula, Bernstein problem, Poincar\'e inequalities,  have been studied (see for example \cite{BSV}, \cite{BCSC}, \cite{JNV},  \cite{CMPS}).
The regularity up to non characteristic boundary points for solutions of  the Kohn Laplacian has been established (see \cite{Jerison}, \cite{BCC}, \cite{BGM}, \cite{CGS}). 
It is worth mentioning that  two possible Laplacian operators have been introduced in  non characteristic planes of $\H ^2$. One, a  sum of squares of vector fields, limit of the Riemannian one, (see \cite{BH} \cite{BBC} for some recent results in general manifolds), and  a second one as a conformal Laplacian, induced by the immersion in the whole space (see \cite{FGMT}), whose fundamental solution is the restriction to the plane of the gauge function.

Much less in known around characteristic points: there is a classical counterexample to boundary regularity at these points (see \cite{Jerison2}), while  curvature is not well  defined at these points.   A clear picture of the geometry of horizontal curves on surfaces with  characteristic points (called $t-$graphs) have been provided in \cite{CHMY}. See also \cite{CU} for some properties of characteristic points, and  \cite{BBCH} for more general results.

Here we consider the plane 
$M=\{(x, 0): x\in \R^4\}$ in $\he 2$, and
 we study its geometry, with the main scope to introduce a Laplace-type operator, whose fundamental solution is a power of the distance induced on the plane by the immersion in $\H^2$. For sake of simplicity, we will denote by $x$ the points of $M,$ instead of $(x,0).$  
 We stress that the only characteristic point of $M$ is the origin $x=0$. At any other point 
we will select an orthonormal basis of the horizontal tangent space $H_xM$, and we will  denote by    $$(Z_i)_{i=1,2,3} \text{ the basis of }H_xM, Z_4 =[Z_2, Z_3], \; \text{ and by } \nu \text{ the horizontal normal to  } M.$$

It is well known that the Heisenberg group law is invariant with respect to rotations on the horizontal plane, so that, also   the distance is  invariant with respect to the action of the  rotations. 
Since the horizontal plane is isomorphic to the plane $M,$
and the rotation group is isomorphic to the quaternionic one, we can identify $M-\{0\},$
with the quaternionic group. This remark can provide some insight on the geometry of $M-\{0\},$ and simplify some estimates, but it does not help to understand the structure of the plane up to the characteristic point. In the sequel we will present a new strategy to treat the characteristic points.

The vector fields $(Z_i)_{i=1,2, 3}$ together with their commutators span the tangent space at all points but the origin. 
Hence, by a well known result of H\"ormander \cite{hormander}, any non vanishing couple of points can be connected with horizontal curve,  integral curve of the vector fields $(Z_i)_{i=1,2, 3}$. In addition all the radial directions are horizontal curves, (we refer to \cite{CHMY} for  properties of characteristic points in general manifolds) so that a Carnot Carath\'eodory distance $d_{cc}$ is well defined on the whole plane $M$ (see \cite{NSW} for the definition). 
Note that the metric is sub-Riemannian at every point different from $0$
and becomes abruptly Euclidean at the origin.
On the other hand,  as already mentioned above, we can define another distance $d$
induced  on the plane by the immersion in $\he 2$, whose balls are the intersection of the horizontal tangent plane with the manifold $M$. The explicit expression of the distance is the following: 
\begin{equation}\label{eq:dist}
d(x, y):= \left(|x-y|^4 +  4 ( 
x_1 y_3- x_3 y_1  + x_2 y_4 - y_2 x_4)^2\right)^{1/4}
\end{equation}
and it does not present any singularity at the origin. 
Note that $d$ and $d_{cc}$ have different geometric meaning: $d_{cc}$ is intrinsic and it explicitly appears while studying curvature problems,   while $d$ seems to be important in the study of boundary value problems (see \cite{Jerison}, \cite{BCC}, \cite{BGM}, \cite{CGS}). In this Note we focus on the properties of the  distance
$d$. 

The plane  has constant homogeneous dimension $Q=5$ up to the origin with respect to the measure $|y|dy$ (see for example \cite{CG}), so that the manifold $M$ is homogeneous and we would like to find a Laplace-type operator whose fundamental solution can be well approximated by a power of the distance $d$.
A Laplacian operator has been defined on the plane in \cite{BH}, \cite{BBCH}  and with our notations is formally defined as follows: 
$$\Delta_M   u= \sum_{i=1}^3\frac{1}{|y|}Z^*_i \Big(|y| Z_i u \Big).$$
Here $Z^*_i$ denotes the adjoint of $Z_i$. 
 However, on non characteristic planes,  as we have already said above, two Laplace operators have been defined: a standard one and  a conformal one, whose fundamental solution is the distance induced by the immersion in the whole space (see \cite{FGMT}). 
In this paper,  being inspired from the second Laplacian, we  introduce (in Section 4.1) a new first order  operator $Z_{4, y-x}$, inspired  also by   the parametrix method: we refer for example to the celebrated work  \cite{RS}, and the papers \cite{C}, \cite{CLM} for a simplified application in  the Heisenberg setting. 
Hence we provide the following definition
$$
\Delta_{x} : = \sum_{i=1}^3\frac{1}{|y|}Z^*_i \Big(|y| Z_i  \Big) + 
\frac{1}{|y|}Z^*_{4, y-x} \Big(|y| Z_{4, y-x}  \Big).$$
Note that this is an operator that can be applied to functions of two entries $x$ and $y$. Operators of this type arise when trying to approximate the parametrix of the fundamental solution in sub-Riemannian setting. 
Also note that the distance computed in \eqref{eq:dist} is locally subriemannian at any point different from $0,$ and it becomes Euclidean at the origin. Formally for $x=0$
also the operator $
\Delta_{x}$ reduces to a Riemannian Laplacian, expressed in terms of all the vector fields $(Z_i)_{i=1, \cdots, 4}.$ The family $(Z_i)_{i=1, \cdots, 4}$  is an orthonormal basis of $\R^4$, so that $\Delta_0$ reduces to the Euclidean Laplacian up to first order terms, at least formally. 
We set now $$\Gamma(x,y) := \frac{1}{C_{\H\times \R}}d^{-Q+2}(x,y),$$ for a purely dimensional  constant $C_{\H\times \R}$ that will be defined later on, see \eqref{CHR}.  In this paper we prove that $\Gamma$ satisfies the following relation.
\begin{theorem}  With the previous notations, we have
$$\Delta_{x}\Gamma(y,x)= -\delta_x + f_x, \ \mathrm{  where }\ |f_x(y) |\leq 
C \sqrt{|x|}d^{-Q+3}(y,0) d^{-Q+1}(x,y),$$ 
\end{theorem}

 Note that   $f$ has two singularities, one at the pole $x$ of $\Gamma$ and one at the characteristic point $0,$ but both are integrable. As $x\to 0$, the two singularities do not collapse in a higher order singularity, but $f_x$ tends to $0$.  Hence $\Gamma$ can be considered a good approximation and parametrix of the fundamental solution of
 $\Delta_x$ up to the characteristic point.

 As a consequence of the explicit expression of the Laplacian and its fundamental solution,  we obtain  a mean value formula   for $C_0^\infty$ functions, which  is the analogous of the standard equality \begin{equation}\label{wow}u(x) =  \int <\nabla_E \Gamma(x,y), \nabla_E u(y)> dy\,,\end{equation} which holds in the Euclidean setting  (here $\nabla_E$ denotes the Euclidean gradient). It is well known that this formula  is the starting point for the regularity theory of  weak solutions in  Sobolev spaces. Here we will adapt this formula to the present setting. Since the operator $\Delta_x$ is in divergence form, we could immediately 
obtain the representation 
$$u(x) = \int \sum_{i=1}^3 Z_i \Gamma (x,y) Z_i  u(y) |y| dy + \int Z_{4, y-x} \Gamma (x,y) Z_{4 ,y-x}  u(y) |y| dy +\int f_x (y) u(y) |y| dy.$$

However this expression does not seem natural, since only 
the vector fields $(Z_i u)_{i=1}^3$ 
belong to the horizontal tangent space, 
while  $Z_{4, y-x}$ is not horizontal.
We need a more sophisticate  formula where only derivatives in the direction $Z_i$ are applied on $u$, while  the derivative in the direction  $Z_{4, y-x}$ is applied only on $\Gamma$.

\begin{theorem} 
There is a constant $C>0$ such that for every $u\in C^\infty_0(M),$ 
$$u(x)= \sum_{i=1}^3 \int (Z_i \Gamma(x,y) + K_i(x,y)) Z_i u(y)| |y|dy + \int \Big(K_0 (x,y) + f_x (y)\Big) u(y) |y| dy,$$
where $K_i$ are suitable kernels such that 
$$|K_0(x,y)|\leq C |x|^2d^{-Q+1}(y,x)d^{-Q+2}(y,0), \quad |K_i(x,y)|\leq C |x|^2d^{-Q+1}(x,y) d^{-Q+3}(y,0),$$
for every $i=1,2, 3$. 
\end{theorem}

If $x=0$, the fundamental solution multiplied by the volume element reduces to the standard Euclidean fundamental solution and the representation formula in Theorem 1.2 reduces to \eqref{wow}.
If $x\not =0$, exactly as the function $f_x$, also the kernels  $K_i$ have two integrable singularities. 
Since the kernel has the standard singularity of the gradient of $\Gamma$ at the pole and a better behavior at the characteristic point, we believe that this result can be the starting point for the study of  the regularity for differential equations  on the characteristic plane and it can pave the way to a deep understanding of differential calculus in this setting.

\section{Generalities on the Heisenberg group and homogeneous structures}

\subsection{Some properties of the Heisenberg group}

In this section we  recall some basic properties of the Heisenberg group $\he 2.$  
The group $\he 2$  is a simply connected 
Lie group of dimension $5$,
whose underlying manifold is $\R^{5} $. 
To simplify notations we will denote $\xn=(x,\tx)$ and $\yn=(y,  \ty) $ the points of $\R^5$, and its  (non commutative) group law will be defined as 
\begin{equation*}\label{glaw}\xn\cdot \yn :=\big(x+y, \tx+\ty - \frac12 \sum_{j=1}^2(x_j y_{j+2}- x_{j+2} y_{j})\big)
\end{equation*}
In this system of coordinates,
the unit element of $\he 2$, can be identified with the zero of the vector space $\R^{5}$, while the inverse is $\xn^{-1}= -\xn$. The Lebesgue measure in $\mathbb R^{5}$ 
is a Haar measure in $\he 2$. For a general review on Heisenberg groups and their properties, we
refer to \cite{libroStein}, \cite{BLU}, \cite{GromovCC} and to \cite{VarSalCou}.
We limit ourselves to fix some notation.

A standard basis of the Lie algebra of the left-invariant vector fields  is given by the vector fields 
$
X_1, X_2, X_3, X_4$ 
defined in \eqref{vf}, and $X_{5} :=
\partial_{\tx}$.
The only non trivial bracket  relations between the vector fields are $
[X_{j}, X_{j+2}] = X_{5}$, for $j=1,2.$ 
The {\it horizontal subbundle}  $H \he 2$ is the subbundle of the
tangent bundle $T \he 2$ spanned by $X_1, \dots, X_{4}.$
Denoting  at every point $\xn$  by $Z_{| \xn} \he 2 $ the linear span of $X_5$ at the point $\xn$, the $2$-step
stratification of  $T \he 2$ is expressed by
\begin{equation*}\label{strat}
T _{| \xn}\he 2 = H _{| \xn}\he 2\oplus Z _{| \xn}\he 2.
\end{equation*}
The vector space $T \he 2$  can be
endowed with an inner product, indicated by
$\scalp{\cdot}{\cdot}{} $,  making
    $X_1,\dots, X_{4}$ orthonormal.
For a general review on Heisenberg groups and their properties, we
refer to \cite{libroStein}, \cite{GromovCC} and to \cite{VarSalCou}.
We limit ourselves to fix some notation.
The Heisenberg group $\he 2$ can be endowed with the homogeneous
norm (known as Kor\'anyi norm) 
\begin{equation}\label{gauge}
\varrho (x, x_{5})=\big(|x|^4+ 16x_{5}^2\big)^{1/4}\,,
\end{equation}
and  the gauge distance (see
 \cite{libroStein}, p.\,638)
\begin{equation*}\label{def_distance}
d(\tilde x,\tilde y):=\varrho ({{\tilde x}^{-1}\cdot \tilde y}),
\end{equation*}
We also  call $Q_\H=6,$ the homogeneous dimension of the space.


\subsection{Order and degree of differential operators}
Let us now consider a differential manifold  $M$, of topological dimension $\tilde N$, with a family of linearly independent vector fields $(X_i)_{i=1, \cdots, m}$ at every point. We will call them  horizontal, and we will  assume that together with their commutators of step 2 they span the  tangent space at every point. In this case we say that the space has step 2, and we will call homogeneous dimension $Q=\tilde N+ 2(\tilde N -m)$. The space become a sub-Riemannian manifold with the choice of a  Riemannian metric on the space spanned by horizontal vector fields.  In this case the   order and degree of operators is defined as follows.
\begin{definition}\label{locdeg}
If $X$ is an horizontal differential operator, we  say that $X$ has order $1$ and degree $1$. If 
 $(X_{i})_{i=1, \cdots, p}$ are horizontal differential operators, 
we say that their composition $Y_1= X_{1}\cdots  X_{p}$ has order $p$ and degree $p$, 
and we  denote by $p= deg(Y_1)$. 
Moreover, if 
 $Y$ is a differential operator  represented as 
\begin{equation}\label{order} 
Y=a Y_1,\end{equation}
where $a$ is a homogeneous function of degree $\alpha$, then 
we say that $Y$ 
is homogeneous of degree $deg(Y_1)-\alpha$. A differential operator will be called of degree  $k-\alpha$ if it is a sum of operators with maximum degree  
$k-\alpha$.
\end{definition}

 Following \cite{Folland} we recall also the definition of kernel of type $\alpha$:
\begin{definition}\label{kerneltype}
We say that a kernel $K$
is of local type $\lambda$ with respect to the distance $d$ if 
for every open bounded set $V$ and
for every $p \geq 0$ there exists a positive constant $C_{p, V}$ such that, for every $\xn, \yn \in V$, with $\xn\not= \yn$ and for every differential operator $Y$ of maximum degree $p$,
\begin{equation*}
	\label{e:sileva}
|YK(\xn, \yn)|\leq C_{p, V}  d(\xn,  \yn)^{\lambda-p+Q}  .\end{equation*}
\end{definition}

\begin{remark}In particular the  fundamental solution $\Gamma$ of a Laplacian is a kernel of
type $2$. 
\end{remark}

\section{The characteristic plane}

In this work we are dealing with the manifold 

$$M=\{(x,0): x\in \R^{4}\} \subset \H^2.$$

A surface $M$, smooth in the Euclidean sense,  can be locally expressed as the zero 
level set of a function $\phi\in C^\infty(\he 2)$, $M=\{\xn\in \he 2: \phi(\xn)=0\}.$
Points where the subriemannian gradient does not 
vanish are called non characteristic, and properties of the solution 
in a neighborhood of these points have been largely studied starting 
from the papers of Kohn and Nirenberg 
in \cite{KN}, and Jerison in \cite{Jerison}
(see also the references therein). 
The horizontal normal vector can be represented as 
$$\nu (\xn)= \frac{\nabla_\H\phi (\xn)}{|\nabla_\H\phi(\xn)| }\,.$$
Points where $\nabla_\H \phi$ vanishes, 
are called characteristic and the geometry of the surface is not completely understood. 
We can give an intrinsic definition of characteristic points as follows. 

\begin{definition}
A point $ x \in M$ is called a {\rm characteristic point } 
if the Euclidean tangent space to $M$ at $x$ at coincides with 
the horizontal tangent plane  $H_x \he 2$ at  $M$ in $x$.
\end{definition}

\subsection{Vector fields tangent to the plane $M=\{(x,0): x\in \R^{4}\}$}

To simplify notation, in the sequel we will simply denote $x$ the points of $M$, in place of $(x,0)$, and in particular $0$ in place of $(0,0)$. 

The horizontal normal $\nu$ to the plane $M$ at $x$ can be computed as 
$$\nu(x) = \frac{(x_3, x_4, - x_1, -x_2)}{d(x,0) }= 
\frac{x_3 X_1 + x_4 X_2 - x_1X_3 -x_2 X_4 }{d(x,0)}.$$

The only characteristic point in  is the origin $x=0$.

%
%
%
%
%

At every not characteristic point $x\in M\setminus\{0\}$, we consider  
$T_1$, $T_2$, $T_3$ the following horizontal  tangent vectors at $x$: 
\begin{equation}\begin{split}\label{fieldsT}
T_1 &=  x_1 X_1 + x_2 X_2 + x_3X_3 + x_4X_4  = x_1 \partial_{x_1} + x_2 \partial_{x_2} + x_3\partial_{x_3} + x_4\partial_{x_4} ,\\
T_2 &=  - x_4 X_1 + x_3 X_2 - x_2 X_3 + x_1 X_4  = - x_4 \partial_{x_1} + x_3 \partial_{x_2} - x_2 \partial_{x_3} + x_1 \partial_{x_4}\,,\\
T_3 &= - x_2 X_1 + x_1 X_2 + x_4 X_3 - x_3 X_4 =
- x_2 \partial_{x_1} + x_1 \partial_{x_2} + x_4 \partial_{x_3} - x_3 \partial_{x_4}\,.
\end{split}
\end{equation}
Note that
\begin{equation}\label{commut23}
[T_2, T_3] = - 2\Big(x_3 \partial_{x_1} + x_4 \partial_{x_2} - 
x_1 \partial_{x_3} - x_2 \partial_{x_4}\Big).\end{equation}
%
We will denote 
$$T_4 := \frac{1}{2}[T_2, T_3],$$
and we choose at every  point $x$ different from $0$ the following 
horizontal orthonormal basis:
$$Z_i:= \frac{ T_i}{|x|}, \forall i= 1, 2, 3, $$
where $|\cdot|$ is  the restriction to $M$
of the Kor\'anyi norm introduced in \eqref{gauge}, and coincides with the standard Euclidean norm in $\R^4$.
Note that 
\begin{equation}\label{commut23Z}
 [Z_2,Z_3]= [\frac{T_2}{|x|}, \frac{T_3}{|x|}] = 
\frac{1}{|x|^2}  [T_2, T_3]  = 
\frac{1}{|x|^2} 2T_4 ,\end{equation}
since $T_2|x| = T_3|x| =0$. Then we  define 
\begin{equation}\label{eq:pippo}
   Z_4 := \frac{1}{|x|^2} T_4. 
\end{equation}

We choose $Z_4$ in such a way that $(Z_i)_{i=1, \cdots , 4}$ is an orthonormal basis, so that $Z_4 = \frac{1}{2}[Z_2, Z_3]$. 
Hence the family $(Z_i)_{i=1, 2, 3}$ satisfies the H\"ormander condition at every point. 

Note also that the normalization in direction $Z_4$ is obtained by dividing by $|x|^2$, due to the fact that it is a commutator. Also note that 
\begin{equation}\label{Znorm}
Z_1 |x| = 1, \; Z_i |x| = 0, \text{  for  }i=2, 3, 4. 
\end{equation}



\begin{definition}
For every point $x$ different from $0$ we  call {\em{horizontal plane at $x$}}, 
$H_{x}M$,
the plane spanned by  
$Z_1, Z_2, Z_3$ and
we  call {\em{horizontal plane  at $0$}} the 
plane $H_{0}M$  spanned by
$\partial_1, \cdots , \partial_4$. 
\end{definition}

We recall that the isometries of the group  of $\H^2$ are rotations on the horizontal plane. We will use the rotations which define the vector fields $T_i$, $i=1, \cdots , 4$. These  
can be naturally associated to the following matrix
of the space $\R^4$
\begin{equation}\label{Ax}
A_x  =\left(\begin{matrix}
x_1  & - x_4 & - x_2 &- x_3\\
 x_2 & x_3  & x_1 &-x_4  \\
 x_3 &-x_2& x_4   &  x_1  \\
 x_4 &x_1 &-x_3 &x_2  
 \end{matrix}\right),
\end{equation}
in the sense that $A_x (\partial_i) = T_i$.   
As a consequence, the vector $Z_i$ will become 
\begin{equation}\label{Zi}
Z_i = A_{\frac{x}{|x|}} (\partial_{x_i}).
\end{equation}

In the following, we shall write
\begin{equation}\label{nul}\origin:= (1,0,0,0).\end{equation}
Using the usual identification of points with their exponential coordinates around $0$, we can apply $A_x$ to points of $M$ and we get

\begin{equation}\label{19 maggio}
A_x \origin = x.
\end{equation}
We notice that
$A^{-1}_x = \frac{1}{|x|^2} A^T_x $, 
Therefore,
\begin{equation}\label{inv}
|x|^2A^{-1}_x  y =\left(\begin{matrix}
    x_1  &  x_2 &  x_3 &x_4 , \\
       - x_4 & x_3  & -x_2 &x_1  \\
      - x_2 &x_1& x_4   &  -x_3  \\
               -x_3 &-x_4 &x_1 &x_2 
 \end{matrix}\right)\left( \begin{matrix}
 y_1\\
y_2 \\
y_3 \\
y_4
 \end{matrix}\right)
=\left( \begin{matrix}
 x_1  y_1 + x_2 y _2 + x_3 y_3+ x_4 y_4
\\
x_1 y_4 - x_2 y_3 +x_3y_2 - x_4y_1 \\
x_1y_2 - x_2 y_1-x_3 y_4 + x_4y_3\\
-x_3y_1-x_4 y_2+x_1y_3+x_2y_4
 \end{matrix}\right) \,.
\end{equation}
In particular, if $x\not=0, $ and we divide every element of the 
last matrix by $|x|^2$, we obtain that 
\begin{equation*}
A^{-1}_{\frac{x}{|x|}}  \frac{y}{|x|} =A^{-1}_x  y.
\end{equation*}
If 
$|x|=1$, the linear transformation $A_x$  is a rotation of the horizontal plane. It is well known that the distance and the measure in $\he 2$ are invariant with respect to these 
rotations, so that their  restriction to the plane $M$ have the same property. 
%
%

\begin{remark}
Let us explicitly recall that the matrix $A_x$
is a representation of a quaternion $x = (x_1, x_2,x_3, x_4) $.
This provides the plane $M$ with a field structure. The sum is the standard componentwise, which coincides with the one in $\R^4$. 
The product of two quaternions is simply 
$$y  \circ x = A_y x  ,
 $$
and the neutral element is 
$n= (1,0,0,0)$.
It is well known that the conjugate of the quaternion is associated to the transpost matrix and that the inverse element, defined as 
$x^{-1}\circ x= x\circ x^{-1} = (1,0,0,0)$, 
can be expressed in term of the conjugate quaternion as follows
$$x^{-1} = \frac{\bar x}{|x|^2}\,.$$
This is nothing but a different notation for expressing assertion \eqref{inv}.

\end{remark}

\subsection{Distance induced on the plane}

We will consider on the plane $M$ the metric induced by the immersion in 
$\he 2$. 
\begin{definition}\label{dist prima volta}
The distance between two points $x$ and 
$y$ in $M$ is defined as 
\begin{equation}\label{distanceM}
d(x, y):= \left(|x-y|^4 +  4 ( 
x_1 y_3- x_3 y_1  + x_2 y_4 - y_2 x_4)^2\right)^{1/4}.
\end{equation}
Here we have denoted by
$x-y$ the Euclidean difference in $\R^4$. Also recall that we denoted by $|\cdot|$ the restriction to $M$
of the Kor\'anyi norm, which coincides with the standard Euclidean norm.
We will denote $ \BalM(x, r)$ the associated sphere.

\end{definition}
Note that the distance can be expressed in terms of the matrix $A_x$ as follows,
\begin{equation}\label{dcirc}
d(x, y) = 
\Big(|x-y|^4 + 4|x|^4((A^{-1}_x y)_4)^2\Big)^{1/4} = 
\Big(|x-y|^4 + 4((A^{T}_x y)_4)^2\Big)^{1/4} \,.
\end{equation}
The expression above is symmetric in $x,y \in M \setminus \{0\},$ so that 
$$
d(x, y) =  
\Big(|x-y|^4 + 4|y|^4((A^{-1}_y x)_4)^2\Big)^{1/4} = 
\Big(|x-y|^4 + 4((A^{T}_y x)_4)^2\Big)^{1/4}.
$$
In particular we notice that the 
dilation of the space restricted to $M$ simply reduce to the standard dilation
$\delta_\lambda x=(\lambda x_1, \lambda x_2, \lambda x_3, \lambda x_4,)
$
and the 
distance is homogeneous of degree one with respect to these dilations; precisely $$ d(\delta_\lambda x, \delta_\lambda y) = 
\lambda  
d(x, y) .$$

From the invariance properties of the space we deduce the following property of the distance.

\begin{remark}\label{francofonia}

Let $x \in M$,  $|x|=1$, then
$$d(A_{x}(y), A_{x}( z)) = d(y, z)\ \forall y, z \in M.$$
For a  general vector $x\in M\setminus \{0\}$ the following equality holds
\begin{equation}\label{distinvariance}
d(A_{x}(y), A_{x}( z)) = |x|d(y, z).
\end{equation}
Indeed the rotations of the horizontal plane are isometries of the Heisenberg group, so that we obtain the first equality. The second one follows from the fact that $d$ is homogeneous of order 1.

\end{remark}

From the previous Remark, we can prove the following result.
\begin{lemma}\label{via} If $x\in M\setminus\{0\}$, we have
$$A_{x}(\BalM(\origin, r))=   \BalM(x, |x| r),$$
where $\origin$ has  been defined in \eqref{nul}.
\end{lemma}

\begin{proof}
If $y \in \BalM(\origin, r)$, then $d(\origin, y)\leq r$. 
Using the fact that $A_{x}\origin = x,$ and \eqref{distinvariance}, it follows that 
$d(x , A_{x}( y)) = |x| d(\origin,  y) \leq |x| r,$
so that 
$A_{x}( y)\in \BalM(x, |x| r)$ and we have the inclusion $A_{x}(\BalM(\origin, r))\subseteq  \BalM(x, |x| r)$. To prove the reverse inclusion, consider $z\in \BalM(x, |x| r)$ and call $y=A_{x}^{-1}( z)$. By \eqref{distinvariance} and \eqref{19 maggio}, 
$$d(y,\origin) = \frac{1}{|x|}d(A_{x}y, A_{x}( \origin))=\frac{1}{|x|}d(z,x)\leq r,$$ by the choice of $z$. It follows that $A_{x}^{-1}(\BalM(x, |x| r)) \subseteq  \BalM(\origin, r)$, and the thesis is proved. 
\end{proof}

The distance has a different behavior at $0$ and at  
non characteristic points. Indeed we have the following result.

\begin{proposition}\label{remapp}
Let us fix a point $x\in M\setminus\{0\}$. 
Note that 
\begin{equation}\label{atrasposta}
(A^{T}_{x} y)^2_4 = (-y_1x_{3}-y_2x_{4}
+y_3x_{1}+y_4 x_{2})^2=  |x|^2((A^{T}_{\frac{x}{|x|}} (y-x))_4)^2 \,,
\end{equation}
so that the coefficient $(A^{T}_{\frac{x}{|x|}} (y-x))_4$
can be interpreted as the fourth component in a coordinate system 
rotated according to $A_{x}$. 
Then the distance of any other point $y$ to the fixed point $x$ is
$$d(y,x)=
\Big(|y-x|^4 + 4|x|^2((A^{T}_{\frac{x}{|x|}} (y-x))_4)^2\Big)^{1/4} \,.
$$

We have 
\begin{itemize}
\item [1.]
if $ |x|^4 \leq ( (A^T_yx)_4)^2 $ it is equivalent to the Euclidean distance;

\item [2.]
if $|x|^4 \geq( (A^T_yx)_4)^2  $ the distance is subriemmannian. More precisely, we can introduce the vector fields 
\begin{equation}
\label{vfi}
A_{\frac{x}{|x|}}(\partial_{1}),\;
A_{\frac{x}{|x|}}(\partial_{2} - \frac{(y-x)_3}{2|x|} \partial_{4}), 
\; A_{\frac{x}{|x|}}(\partial_{3} + \frac{(y-x)_2}{2|x|} \partial_4)\,.
\end{equation}
Denote $d_{x}$ the induced distance. Then the following relation holds:
$$\frac{1}{2}d(x,y) \leq d_{x}(x,y)\leq 2 d(x,y) \,.$$
\end{itemize}

\end{proposition}

\begin{proof}
It holds that 
$$|y-x|^2 = \sum_{i=1}^4 ((A^{T}_{\frac{x}{|x|}} (y-x))_i)^2,$$
and the distance is equivalent to 
$$d^4(y,x)\approx
\Big(\sum_{i=1}^3 ((A^{T}_{\frac{x}{|x|}} (y-x))_i)^2 \Big)^2 + ((A^{T}_{\frac{x}{|x|}} (y-x))_4)^4 + |x|^2 
 ((A^{T}_{\frac{x}{|x|}} (y-x))_4)^2 .$$
 The behavior of the distance is Euclidean if 
$$  ((A^{T}_{\frac{x}{|x|}} (y-x))_4)^2 \geq  |x|^2 .$$
 By \eqref{atrasposta} this means 
$$  (A^{T}_{x} y)_4^2 \geq  |x|^4 .$$
If this is not satisfied, and $  |x|^4 \geq  (A^{T}_{x} y_4)^2$, the distance is equivalent to a the subriemannian one, in a coordinate system rotated by $A^T_{\frac{x}{|x|}}$. Then, according to \eqref{Zi} the induced rotation on the vector fields is the one contained in \eqref{vfi}. 
\end{proof}

%
%
%

\subsection{The induced measure }
 
The measure induced on the plane $M$ by the immersion can be written 
$$\mu_M(E) = \int_E |x| d x$$
for every measurable set $E\subset M$ (see \cite{CDG94} and \cite{FSSC96}). 
From \eqref{Znorm} it follows that the density function of the measure $\mu_M$ is regular as a function of the vector fields $(Z_i)_{i=1, 2, 3}$.  

We recall the following result contained in \cite{CG}. We provide here the proof for reader convenience.
\begin{lemma}\label{misurasfera}
The distance defined on $M$ satisfies the doubling condition with respect to the 
metric $\mu_M$ defined above. Precisely there exist two constants 
$C_1$ and $C_2$, independent of $x_0$, such that
$$C_1 r^5\leq \mu_M(\BalM(x, r)) \leq C_2 r^5$$
for every $x$ and for every $r$.
\end{lemma}

\begin{proof}
Let us first assume that $x=0$.
In this case $d(y,0)= |y- x|$ in Euclidean sense. Hence, in this case $\BalM(x, r)$ concides with the Euclidean ball $B_E(x, r)$ of $\R^4$. 

$$\mu_M(\BalM(x, r))= \int_{B_M(x, r)} |y| d y = C r^5$$

If $x\not=0$, we first note that we can replace the ball with an  equivalent  cube: 
\begin{align*} Q(x, r)  := &\big\{y: |(A^{T}_{\frac{x}{|x|}}(y-x))_i|<r,\ i=1,2,3,4;  \  |x|\cdot|(A^{T}_{\frac{x}{|x|}}(y-x))_4|< r^2\big\}\\ 
=& \big\{y: |(A^{T}_{\frac{x}{|x|}}(y-x))_i|<r,\ i=1,2,3;  \ |(A^{T}_{\frac{x}{|x|}}(y-x))_4| < \min \Big(r, \frac{r^2}{|x|} \Big)\big\}.\end{align*}
Indeed, 
there are absolute constants $C_1$ and $C_2$  such that 
$$Q(x, C_1 r) \subseteq \BalM(x, r)\subseteq Q(x, C_2 r).$$
Hence we can equivalently estimate the measure of the cube or of the ball.

If $\frac{r^2}{|x|}\leq r$,  then $r\leq |x|$, and for every $y\in \BalM(x, r),$ $|y|\leq |x| + r\leq 2|x|$, so that   
$$\mu_M(\BalM(x, r))= \int_{\BalM(x, r)} |y| d y \leq 2 |x|\int_{\BalM(x, r)} d y \leq 
2|x|\,|Q(x, C_2r)|\leq 
Cr^5.$$
On the other side, for all $y\in \BalM(x, r/2),$ we have $|y|\geq |x|- |y-x| \geq  |x|-r/2\geq |x|/2,  $ so that 
$$\mu_M(\BalM(x, r))\geq \mu_M(\BalM(x, r/2)) \geq \frac{|x|}{2} \int_{\BalM(x, r)}  d y = c r^5. $$

If $r \leq \frac{r^2}{|x|}$, then $|x|\leq r$ and $\BalM(x, r)$ is the Euclidean sphere $B_E(x, r)$ so that
$$\mu_M(\BalM(x, r))\leq r |B_E(x, r)| = Cr^5.$$
To estimate $\mu_M(\BalM(x, r))$ from below, we first assume that 
 $|x|\leq r/2,$ so that $B_E(0, r/2) \subset B_E(x, r) = \BalM(x, r)$ and we have 
 $$\mu_M(\BalM(x, r))\geq \int _{B_E(0, r/2)}|y| dy = Cr^5.$$
Finally, if $r/2 \leq |x|\leq r,$
$$\mu_M(\BalM(x, r))\geq  r |B_E(x, r)|\geq  c r^5\,.$$
\end{proof}

Thanks to Lemma \ref{misurasfera} we can say that the plane $M$ has local homogeneous dimension 5, and we will always denote it 
$$\dimM=5.$$

\begin{proposition}\label{stimaconv}
The function $$y\mapsto  \frac{d(x, y)}{\mu_M(\BalM(x, d(x, y)))}$$
is locally integrable with respect to the measure $\mu_M$. 
For every $U$ open bounded in $\R^4$, there exists a constant $C>0 $ such that for every $x\in U$ 
\begin{equation}\label{dQ1}\int_{\BalM(x, R)} \frac{d(x, y)}{\mu_M(\BalM(x, d(x, y)))} d\mu_M(y) \leq CR, 
\end{equation}
and, for every
 $y, x \in U$, 
$$
\int \frac{d(x, z)}{\mu_M(\BalM(x, d(x, z)))}
\frac{d(z, y)}{\mu_M(\BalM(y, d(z, y)))}
d\mu_M(z) \leq C \frac{d^{2}(x, y)}{\mu_M(\BalM(x, d(x, y)))}.$$

\end{proposition}
We omit the proof, which is standard and only relies on the fact that $\mu_M$ is doubling (see for example \cite{CGL} Proposition A4).



\section{A Laplace operator on the plane and its fundamental solution}

\subsection{Homogeneity properties of the distance function and its powers}

Since the space $M$ has homogeneous dimension $Q=5$ with respect to the measure $|x| dx$, it is natural to define the function   
\begin{equation}\label{fond}\Gamma(y, z):=  \frac{1}{C_{\H\times \R}}d^{-3}(y, z)= \frac{1}{C_{\H\times \R}}d^{-Q+2}(y, z), \end{equation}
where  $C_{\H\times \R}$ is a normalization constant, computed in Lemma \ref{singular}, which can be explicitly written as
\begin{equation}\label{CHR}C_{\H\times \R}:= 21 \int_{ (x^2_1 + x^2_2 + x_4^2)^2 + x_3^2\leq 1}
  \Big( x^2_2 +x^2_4 + x^2_1\Big)dx.
\end{equation}

The aim of this section is to show that $\Gamma$
can be considered as a fundamental solution of the  Laplacian defined in \eqref{operator}.

Due to the invariance properties of the distance, 
 see Remark
\ref{francofonia}, we have the following property of $\Gamma$. 
\begin{remark}\label{invariancegamma}
Let $x \in M$,  $|x|=1$, then
$$\Gamma(A_{x}(y), A_{x}( z)) = \Gamma(y, z)\ \forall y, z \in M.$$
\end{remark}

This remark tells us that, when studying the behavior of the function $\Gamma$, we can assume that 
$$z= (z_1, 0, 0, 0),\quad \mathrm{where}\ \ 
 z_1\in \R.$$

In the sequel we will need a special first order operator. According to \eqref{Ax}, the vector fields $T_{j}$, $j=1,\cdots, 4$,  can be represented as $T_{j} =  \sum_{i=1}^4 (A_y)_{ij}\partial_{y_i}$. Moreover, we write 
\begin{equation}T_{4, z}: =  \sum_i (A_z)_{i4}\partial_{y_i},\end{equation}
to stress that the variable of the
coefficient can be different from the variable with respect to we derive. 
In addition,  we call 
\begin{equation}\label{TZ} Z_{4, y-z}:= \frac{T_{4, y-z}}{|y|}\,.\end{equation}
Here we normalized $Z_{4, y-z}$ by dividing by ${|y|},$ instead of $|y|^2$ as in $Z_4$, see \eqref{eq:pippo}, since we want to define a derivative with the homogeneity of a first derivative. 
Since the vector field depends on two variables, sometimes we will write also $Z^y_{4, y-z}$ to indicate that the derivatives are taken with respect to the variable $y$. However if no superscipt is specified we will intend that the  derivative is taken with respect to the variable $y.$  

\begin{remark}\label{42} Let $z=(z_1, 0, 0, 0).$
We have already computed the derivatives of $|y|$ in \eqref{Znorm}:
$T_1 |y| =|y|, T_{i}|y|= 0 \text{  for every } i=2,3,4$. In 
addition, 
$$T_{4, y-z}|y| = - \frac{z_1 y_3}{|y|}. 
$$

\end{remark}

\begin{lemma}Assume that  $z=(z_1, 0, 0, 0).$
Then, for  $i=1, 2, 3$, 
\begin{equation}\label{ZGammai}|Z_{i} \Gamma(y,z)|  \leq 8 \frac{  d(y,z) + |z| }{C_{\H\times \R}|y|}d^{-Q+1}(y,z)\,.\end{equation}
\end{lemma}
\begin{proof}
Using Remark \ref{42}, we have  $T_{1, y-z}|y-z|^4 = 4 |y-z|^3T_{1, y-z}|y-z|= 4 |y-z|^4 $, 
\begin{align}
T_{1 } d^4(y, z) & 
= T_{1, y-z} |y-z|^4 + 
T_{1. z} |y-z|^4 + 4 T_1 (z_1^2 y_3^2)\nonumber \\ 
&= 4   |y-z| ^4 +  4 |y-z|^2 z_1 (y-z)_1 +8y_3^2z_1^2\nonumber
\\
T_{2}d ^4(y,z) & =   4 |y-z|^2 z_1 y_4 - 8 y_2 y_3 z_1^2,  \label{Td4}  \\
T_{3} d^4(y,z)  & =   4 |y-z|^2 z_1 y_2 + 8 y_4 y_3 z_1^2.\nonumber
\end{align}

It follows that 
\begin{align}
T_{1}\Gamma(y,z) & = 
-\frac{3}{2C_{\H\times \R}} d^{-7}T_1 d^{4}(y,z) \nonumber \\ &= - 
\frac{3}{C_{\H\times \R}} d^{-7} 
\big(     |y-z| ^4 +  |y-z|^2 z_1 (y-z)_1 +2y_3^2z_1^2 
\big),\nonumber \\
T_{2}\Gamma(y,z) & =-\frac{3}{C_{\H\times \R}}d^{-7}(y,z)\Big( |y-z|^2 z_1 y_4 - 2 y_2 y_3 z_1^2\Big),\label{TGamma}\\
T_{3}\Gamma(y,z) & =-\frac{3}{C_{\H\times \R}}d^{-7}(y,z)\Big( |y-z|^2 z_1 y_2 + 2 y_4 y_3 z_1^2\Big),\nonumber 
\end{align}
and the thesis follows 
recalling that $Z_i = \frac{T_i}{|y|}$, for $i=1,2,3$.
\end{proof}

\begin{remark}
We  expect that the derivative with respect to $Z_4$ has the homogeneity of a second order derivative. According to \eqref{TZ} the vector field $Z_{4 ,y-z}$ has to be considered as an operator of degree $1$. Indeed the following estimate holds:
\begin{align}
|Z_{4, y-z} \Gamma(y,z)| & \leq 3 d^{-Q+1}(y,z)\frac{ |z|}{C_{\H\times \R}|y|}.\label{ZGamma4yz}\end{align}
 
\end{remark}
\begin{proof}
As before the proof is a simple computation. Indeed,
\begin{align}
\label{T4yzd4}
 T_{4, y-z} d^4(y,z) & =  8   (y-z)_1 y_3 z_1^2\,.
\end{align}
Consequently, \begin{align}\label{venerdi}
Z_{4, y-z} \Gamma(y,z) & = 
 - 6\, d^{-7}(y,z)   (y-z)_1 \frac{y_3 z_1^2}{C_{\H\times \R}|y|}. 
\end{align}
The thesis immediately follows keeping in mind the expression of the distance $d$, see \eqref{distanceM}.

\end{proof}

 Clearly we are interested to estimate the behavior of $\Gamma$ around the pole. Hence, we can assume that $|z|\leq 2|y|$.

 The following corollary immediately follows from the estimates \eqref{ZGammai} and \eqref{ZGamma4yz}.
\begin{corollary}
If $|z| \leq 2 |y|$,  then
\begin{align}|& Z_{i} \Gamma(y,z)|  \leq 64 \, d^{-Q+1}(y,z) \quad \text{for} \; i= 1, 2, 3\\|&Z_{4, y-z} \Gamma(y,z)|  \leq 6\, d^{-Q+1}(y,z),
\end{align}
where  $Q=5$ is the homogeneous dimension with respect to the measure $d\mu(y) = |y| dy$,
\end{corollary}

\subsection{A non local Laplacian on the plane }

We introduce here a Laplace-type operator  on the plane $M$.

As explained in the Introduction, in \cite{BH} a variational Laplace-Beltrami operator $\Delta_M$ has been introduced on the manifold $M$. If we consider  
its formal weak definition as 
$$\Delta_M u = f \Leftrightarrow \sum_{i=1}^3\int_{\R^4} Z_i u Z_i h |y| dy = \int_{\R^4} f  h |y| dy$$
for every $h\in C_0^\infty(M)$,
we can say that its  expression is 
$$\Delta_M = \sum_{i=1}^3\frac{1}{|y|}Z^*_i \Big(|y| Z_i \Big), $$
where $Z^*$ is the adjoint of $Z_i$ .

We recall that in non characteristic planes two Laplace operators can be defined. A first one analogous to this one, and  a second one, defined in \cite{FGMT}
 (see also \cite{GT}), and called conformal Laplacian. 

We will introduce here an operator analogous to the conformal Laplacian in \cite{FGMT} on the characteristic plane. The definition is also inspired by the expression of derivatives in \eqref{TGamma}: 
 the derivatives $T_2\Gamma$ and  $T_3\Gamma$ have similar expression, while 
$T_1\Gamma$ is uncoupled. We consider the following operator
\begin{equation}\label{operator}
\Delta_{z} := \sum_{i=1}^3\frac{1}{|y|}\,Z^*_i \Big(|y| Z_i  \Big) + 
\frac{1}{|y|}Z^*_{4, y-z} \Big(|y| Z_{4, y-z}  \Big).\end{equation}

Note that this operator depends explicitly on $z,$ due to the presence of the operator $Z_{4, y-z}.$
Due to the lack of invariance in one direction, it will be a  non local Laplacian operator, and it will be possible to apply it only to kernels of the two variables $y,$ $z$. Also the value of the Laplacian will depend on these two variables.  

Let us compute explicitly the adjoint operators, to clarify the expression in \eqref{operator}.
\begin{lemma}\label{aggiunti}
The following equalities hold:
$$Z_1^* = - Z_1 - \frac{3}{|y|}; \quad  Z_i^* = - Z_i\  \text{ for } i=2,  3; \quad Z^*_{4, y-z} = - Z_{4, y-z} + \frac{z_1y_3}{|y|^3}. $$
\end{lemma}

\begin{proof}
Let us start with the first vector field, using Remark \ref{42} and the explicit expression of $T_1$, as given in \eqref{fieldsT}. If $f, g \in C_0^\infty(M)$,
\begin{align*}\int Z_{1} f(y) g(y)  dy =& \int \frac{1}{|y|}T _{1} f(y) g  (y)  dy \\ =& 
- 4\int \frac{ f (y)g(y)}{|y|}    dy + \int \frac{ f (y)g(y)}{|y|}    dy -  \int  f(y) Z_{1}g (y) dy. 
\end{align*}

The proof is analogous for $i=2, 3.$ 

Let us consider $Z_{4, y-z}$. 
As before we assume $z=(z_1,0,0,0).$ Then, keeping again in mind Remark \ref{42}, we have:
\begin{align*}&\int Z_{4, y-z} f (y)g (y)dy = \int T_{4, y-z} f (y)\frac{1}{|y|}g (y)   dy \\  
=& - \int \Big((y-z)_3 \partial_{y_1} + (y-z)_4 \partial_{y_2} - 
(y-z)_1 \partial_{y_3} - (y-z)_2 \partial_{y_4}\Big) f(y)  \frac{1}{|y|}g(y) dy \\ =& -\int f (y)Z_{4, y-z}  g (y) dy  
+ \int\frac{z_1y_3}{|y|^3}f(y)g(y)dy.
\end{align*}

\end{proof}

Replacing in \eqref{operator} the expression of $Z_i^*, Z^*_{4 ,y-z}$  obtained in the previous lemma, it 
follows that 
\begin{equation}\label{operator1}
\Delta_{z} u = - \sum_{i=1}^3\frac{1}{|y|^2}T^2_i u   
- \frac{3}{|y|^2}T_1 u
-
\frac{1}{|y|^2}T^2_{4, y-z} 
u + \frac{ z_1 y_3}{|y|^4}T_{4 ,y-z}
u\,.\end{equation}

Having in mind the definition of $\Gamma$ given in \eqref{fond}, our aim is now to compute  $\Delta_z\Gamma(y,z).$
In Proposition \ref{solfond} we will show that $\Gamma$ is an approximation of the fundamental solution of $\Delta_z$ and that $\Gamma$ is locally integrable, both for $z\not=0$ and for $z=0$. 
Notice that on  the characteristic plane $M$  we need to control both the pole and the  characteristic point $0$, when dealing with the estimates on the derivatives of $\Gamma$.

We will start with the estimate of the gradient square of $d^4(y,z)$ when $z\neq 0$.

\begin{lemma}\label{gradGammasquare}Assuming $z=(z_1,0,0,0)$, we have:
$$\sum_{i=1}^3(T_{i } d^4 (y,z))^2 + (T_{4 ,y-z} d^4(y,z))^2  = 
R_{1,1} (y,z)+  R_{1,2}(y,z) + R_{1,3}(y,z),$$
where
$$R_{1,1}(y,z) = 16 d^4(y,z) z^2_1\Big( (y-z)^2_1+y^2_4  + y^2_2
\Big),\quad 
R_{1,2} (y,z)= 4\Big(   |y-z| ^4  + d^4(y,z)\Big)^2
$$
are the leading terms, respectively for $z\not=0$ and for $z=0$, while 
$$R_{1,3}(y,z)= 16\Big(    |y-z| ^4  + d^4(y,z)\Big)  |y-z|^2 z_1 (y-z)_1 
$$
is a higher order term, satisfying $|R_{1,3}(y,z)| \leq 16 |z| d^{7}(y,z).$
\end{lemma}
\begin{proof}

We can simply sum the first term already computed in \eqref{Td4} with the one in \eqref{T4yzd4}, also using the fact that 
$
2   |y-z| ^4 +  2 |y-z|^2 z_1 (y-z)_1 +4y_3^2z_1^2 =  |y-z| ^4 + d^4(y,z)+ 2 |y-z|^2 z_1 (y-z)_1$,

\begin{align*}(T_{1}&d ^4(y,z) )^2  +  (T_{4, y-z}d ^4(y,z) )^2   \\ =& \Big(2   |y-z| ^4  + 2d^4 + 4 |y-z|^2 z_1 (y-z)_1 \Big)^2 + 8^2   (y-z)^2_1 y^2_3 z_1^4 \\   =& 4\Big(   |y-z| ^4  + d^4(y,z)\Big)^2  + 16\Big(    |y-z| ^4  + d^4(y,z)\Big)  |y-z|^2 z_1 (y-z)_1  \\ &  +16 d^4(y,z) z^2_1 (y-z)^2_1.
\end{align*}
Then we   sum the second and third terms  computed in \eqref{Td4},
\begin{align*}(T_{2}d ^4 )^2 & +  (T_{3}d ^4 )^2=   16 |y-z|^4 z_1^2 (y^2_4  + y^2_2)+64  (y^2_2 + y^2_4 )  y^2_3 z^4_1\\ & =16 z_1^2 (y^2_4  + y^2_2)   (  |y-z|^4 + 4  y^2_3z_1^2)=16 z_1^2 d^4 (y^2_4  + y^2_2). \end{align*} 
Summing the previous equalities we obtain
$$(T_{1 } d^4)^2 + (T_{4 y-z} d^4)^2  + (T_{2}d ^4 )^2 +  (T_{3}d ^4 )^2=  R_{1, 1} + R_{1,2} + R_{1,3},$$
where $R_{1, i}$ are defined in the statement of the theorem. 
\end{proof}

The next step is to apply the principal part of the Laplace type operator to the function $d^4$.

\begin{lemma}\label{Laplad4}
The following equality holds: $$\sum_{i=1}^3T^2_i   (d^4(y,z)) + 
T^2_{4 y-z}    (d^4(y,z)) =  
R_{2,1}(y,z) + R_{2,2}(y,z) + R_{2,3}(y,z)\,,
$$
where 
$$R_{2,1}(y,z) = 28  z^2_1 \Big( y^2_2 +y^2_4 + (y-z)^2_1\Big), \quad  
R_{2, 2}(y,z)=16  |y-z|^4 $$
are the leading terms respectively for 
$z\not=0$ and for $z=0,$
while $$R_{2,3}(y,z) = 
4 y_3^2 z^2_1 +
36 |y-z|^2  z_1 (y-z)_1
$$
is a higher order term,  satisfying $|R_{2,3}(y,z)| \leq 36 (|z|  + \sqrt{|y| |z|})d^{3}(y,z).$
\end{lemma}

\begin{proof}
The proof is a direct computation. Indeed
\begin{align*}T^2_{1 }  d ^4 & = T_1\Big(4 |y-z|^4  + 8y_3^2z_1^2+ 4 |y-z|^2 z_1  (y-z)_1\Big)\\ \nonumber& = 16 |y-z|^2 \sum_{i=1}^4 y_i (y-z)_i + 16y_3^2z_1^2 + 8 \sum_{i=1}^4 y_i( y-z)_i  z_1  (y-z)_1 + 4 |y-z|^2 z_1y_1  \\ \nonumber
&= 16 |y-z|^4 + 28 |y-z|^2z_1(y-z)_1 + 16y_3^2z_1^2 +8 (y-z)_1^2 z^2_1 + 4 |y-z|^2 z^2_1,
\end{align*}
and
\begin{align*}
T^2_{4, y-z} d^4 &=  8 T_{4, y-z}\Big(  (y-z)_1 y_3 z_1^2\Big) = 8(y-z)^2_1  z_1^2-8 y_3^2 z_1^2. 
\end{align*}
In addition,
\begin{align*}T^2_{2}d ^4 & =   4 T_2 |y-z|^2 z_1 y_4  +  4 |y-z|^2 z_1 T_2 y_4  - 8 T_2\Big(y_2 y_3 z_1^2\Big)\\& \nonumber= 4 T_{2, y-z} |y-z|^2 z_1 y_4 +4 T_{2, z} |y-z|^2 z_1 y_4  +  4 |y-z|^2 z_1 T_2 y_4  - 8 T_2\Big(y_2 y_3 z_1^2\Big)=
\\ & =  8 z^2_1y^2_4  +  4 |y-z|^2 z_1 (y-z)_1+  4 |y-z|^2 z_1 ^2   - 8 (y_3^2 - y^2_2 )z_1^2,\nonumber\\
T^2_{3}d ^4 & =  8 z^2_1y^2_2  +  4 |y-z|^2 z_1 (y-z)_1+  4 |y-z|^2 z_1 ^2 +  8 (y_4^2 - y^2_3 )z_1^2.\end{align*}
Summing these  last expressions,  we get 
\begin{align*}&(T^2_{1} +
T^2_{4, y-z} )d^4+ 
(T^2_{2} + T^2_{3})d ^4  \\  = & 16  z^2_1 \Big( y^2_2 +y^2_4 + (y-z)^2_1\Big) + 12 |y-z|^2 z^2_1 +
16 |y-z|^4 + 36  |y- z|^2 z_1  (y-z)_1 
- 8 z_1^2 y^2_3 \\ =&
28  z^2_1 \Big( y^2_2 +y^2_4 + (y-z)^2_1\Big) + 4 y_3^2 z^2_1 +
16 |y-z|^4 + 36  |y- z|^2 z_1  (y-z)_1 
 \\ =&
R_{2,1} + R_{2,2} +R_{2,3},
\end{align*}
 with $R_{2,i}$ as  in the statement of the lemma. 
\end{proof}

Now we can evaluate the operator $\Delta_z$ on $\Gamma$.  As expected the  leading terms $R_{i, 1}$ and 
$R_{i, 2}$ will vanish, and the Laplacian will be estimated in terms of the lower order terms $R_{i,3}$.
It  will be locally integrable, both for $z\not=0,$ and for $z=0.$

\begin{lemma}Let us choose a point $z$ of 
the form $z=(z_1,0,0,0)$, and let us call 
\begin{equation*}f_z(y):=\Delta_z \Gamma(y,z),\end{equation*} for every $y\not= z$. Then 
$f_z$ can be explicitly written as 
\begin{align}\label{fz}
f_z(y) = 
\frac{3d^{-11}(y,z)}{4C_{\H\times \R}|y|^2}\Big( 
 -\frac{7}{4} R_{1,3}(y,z) + d^4 (y,z) R_{2,3}(y,z) + d^4 (y,z) 
 R_{3,3}(y,z) -  112 y_3^4 z_1^4 \Big),
\end{align}
where 
$R_{1,3}$, and $R_{2,3}$ are defined in Lemma  \ref{gradGammasquare} and  Lemma \ref{Laplad4}, while 
$$R_{3,3}(y,z) :=    4 |y-z|^2 z_1 (y-z)_1 +8y_3^2z_1^2 
- \frac{ 8z_1 y_3   }{|y|^2}   (y-z)_1 y_3 z_1^2.
$$
In addition, $f_z(y)$ is locally integrable (in the variable $y$) and it satisfies the following estimate
\begin{align}\label{DeltaGamma}
|f_z(y)|
\leq
\frac{\sqrt{|z||y|}+ |z|}{C_{\H\times \R}|y|^2}d^{-Q+1}(y,z),
\end{align}
for every $y\not= z$.

\end{lemma}

\begin{proof}
Call $T_5 = T_{4, y-z}$, and recall that $\Gamma = \frac{1}{C_{\H\times \R}} d^{-3}$, see \eqref{fond}. 
Then
\begin{align*}\Delta_z \Gamma & = \frac{3}{4C_{\H\times \R}|y|^2} \Bigg( \sum_{i\not=4}^5 T_i\big( 
 d^{-7}   T_id^{4} \big)  + 3 
 d^{-7}   T_1d^{4} 
  - \frac{ z_1 y_3 d^{-7} }{|y|^2}T_{4 y-z}d^{4}
  \Bigg)=\\ & = \frac{3d^{-11}}{4C_{\H\times \R}|y|^2}\Bigg(\sum_{i\not= 4} \Big(
 -\frac{7}{4}  (T_i d^4)^2 +   d^{4}  T_i^2 d^4 \Big)+ 3 
 d^{4}   T_1d^{4} 
  - \frac{ z_1 y_3 d^{4} }{|y|^2}T_{4 y-z}d^{4} \Bigg). \nonumber
\end{align*}
By Lemma \ref{gradGammasquare} and  Lemma \ref{Laplad4} we obtain 
$-\frac{7}{4} R_{1,1}  + d^4 R_{2,1} =0,
$
so that 
$$
\sum_{i\not= 4} \Big(
 -\frac{7}{4}  (T_i d^4)^2 +   d^{4}  T_i^2 d^4 \Big)=  -\frac{7}{4}\big( R_{1,2} + R_{1,3}\big) + d^4\big( R_{2,2} + R_{2,3}\big). $$
Using the expression \eqref{Td4} and \eqref{T4yzd4} of the derivatives of $d^4,$ we have  
$$ 3 
   T_1d^{4} 
  - \frac{ z_1 y_3  }{|y|^2}T_{4 y-z}d^{4} =    R_{3,2} + 
 R_{3,3}, $$
 where \begin{equation*}
 R_{3,2} := 12 
  |y-z| ^4, \quad R_{3,3} :=    4 |y-z|^2 z_1 (y-z)_1 +8y_3^2z_1^2 
- \frac{ 8z_1 y_3   }{|y|^2}   (y-z)_1 y_3 z_1^2.\end{equation*}
With these notations 
$$\Delta_z \Gamma= 
\frac{3d^{-11}}{4C_{\H\times\R}|y|^2}\Big( 
 -\frac{7}{4}\big( R_{1,2} + R_{1,3}\big) + d^4\big( R_{2,2} + R_{2,3} +  R_{3,2} + 
 R_{3,3}\big)\Big).
 $$
We note that \begin{align*}
-\frac{7}{4} R_{1,2} &+ d^{4 }R_{2,2} +d^{4 }R_{3,2} = 
-  7\Big(   |y-z| ^4  + d^4\Big)^2 +  16 d^4 |y-z|^4 +
 12  d^4 |y-z| ^4 =  \\& = - 7\Big(   |y-z| ^4  - d^4\Big)^2 =
  - 7 \cdot 16 y_3^4 z_1^4.
\end{align*}
It follows that $$\Delta_z \Gamma= 
 \frac{3d^{-11}}{4C_{\H\times \R}|y|^2}\Big( 
 -\frac{7}{4} R_{1,3} + d^4  R_{2,3} + d^4  
 R_{3,3} -  7 \cdot 16 y_3^4 z_1^4 \Big). 
 $$
Consequently 
$$|\Delta_z \Gamma| \leq \frac{z^2_1y_3^2}{C_{\H\times \R}|y|^2} d^{-7} + |z| \frac{d^{-4}}{C_{\H\times \R}|y|^2}\leq 
\frac{\sqrt{|z||y|}+ |z|}{C_{\H\times \R}|y|^2}d^{-4}.
$$
\end{proof}

\subsection{The approximate fundamental solution of the Laplace operator}

Up to here we estimated the Laplacian of $\Gamma$ far from the pole. In order to compute the Laplacian in the distributional sense, we need to start with an  estimate on the  boundary of a ball. 
Let us recall that the vector fields $Z_1, Z_2, Z_3, Z_{4, y-z}$ are a basis out of the origin. When a vector $v$ is represented with respect to this basis,  we will denote 
$v_{Z_1},$
$v_{Z_2},$
$v_{Z_3}$, 
$v_{Z_{4, y-z}}$ its components. 

\begin{lemma}\label{singular}
The function $\Gamma$ is defined in terms of a constant $C_{\H\times \R},$ defined in \eqref{CHR}. Then 
$$\int_{d(z,y) =\e}
\sum_{i=1}^3
Z_i \Gamma (y,z) 
\nu_{Z_i} |y| d\sigma(y) 
+ \int_{d(z,y) =\e}
Z_{4 y-z} \Gamma (y,z)\nu_{Z_{4 y-z}}
|y| d\sigma(y) 
\to -1$$
as $\e \to 0$, where $\nu$ is the external normal vector of the sphere.
\end{lemma}
\begin{proof}

We will assume that $z_1\not=0$
To simplify notations we will indicate
$$I= \int_{d(z,y) =\e}
\sum_{i=1}^3
Z_i \Gamma (y,z) 
\nu_{Z_i} |y| d\sigma(y) 
+ \int_{d(z,y) =\e}
Z_{4, y-z} \Gamma (y,z)\nu_{Z_{4, y-z}}
|y| d\sigma(y) .$$
Using the expression of $\Gamma$, i.e.  $\Gamma= \frac{1}{C_{\H\times \R}}d^{-3}$, and integrating by parts we obtain
\begin{align*}I=& -\frac{3}{4C_{\H\times \R}} \e^{-7}\int_{d(z,y) =\e}
\sum_{i=1}^3
Z_i d^4 (y,z) 
\nu_{Z_i} |y| d\sigma(y) \\ &
-\frac{3}{4C_{\H\times \R}} \e^{-7} \int_{d(z,y) =\e}
Z_{4, y-z} d^4 (y,z)\nu_{Z_{4, y-z}}
|y| d\sigma(y) \\ = & -\frac{3}{4C_{\H\times \R}} \e^{-7}\int_{d(z,y) \leq\e}
\Delta_z d^4(y,z) |y| dy.\end{align*}
We can insert here the explicit expression of the Laplacian, and we find
$$I = - \frac{3}{4C_{\H\times \R}} \e^{-7}\int_{d(z,y) \leq\e}
\Big(\sum_{i=1}^3T^2_i d^4  
+ 3T_1 d^4
+
T^2_{4, y-z}  d^4\Big)\frac{1}{|y|} dy.
$$
Now we apply the explicit expression of the derivatives of $T_1d^4,$ obtained in \eqref{Td4} and \eqref{Laplad4}, and we obtain
\begin{align*}I = & - \frac{3}{4C_{\H\times \R}} \e^{-7}\int_{d(z,y) \leq\e}
\Big(28  z^2_1 \Big( y^2_2 +y^2_4 + (y-z)^2_1\Big)+16|y-z|^4 + 
4 y_3^2 z^2_1\\  & 
 +
36 |y-z|^2  z_1 (y-z)_1 + 6\big(   |y-z| ^4 + d^4+ 2 |y-z|^2 z_1 (y-z)_1\big)\Big)
\frac{dy}{|y|}\\ = & -\frac{3}{4C_{\H\times \R}} \e^{-7}\int_{d(z,y) \leq\e}
\Big(28  z^2_1 \Big( y^2_2 +y^2_4 + (y-z)^2_1\Big)+
28|y-z|^4 + 
28 y_3^2 z^2_1 \\ & +
48 |y-z|^2  z_1 (y-z)_1 \Big)
\frac{dy}{|y|}. 
\end{align*}
We apply the change of variables 
$(y-z)_1 = \e x_1, \; \; y_2 = \e x_2, \; \; 
y_3 = \frac{\e^2}{|z_1|} x_3, \; \;  y_4 = \e x_4.$
Let us call $$||x||_{\e, |z_1|} :=   (x^2_1 + x^2_2 + \frac{\e^2}{z_1^2} x_3^2 + x_4^2)^{1/2}\;\; \text{and }\;\;
|x|_\e:=\sqrt{|\e x_1 + z_1|^2 + \e^2 x_2^2 + \frac{\e^4}{z^2_1} x^2_3 + \e^2 x_4^2}. 
$$
Then we get
$$I = -\frac{3}{4C_{\H\times \R}} \int_{ ||x||^4_{\e, |z_1|} + x_3^2\leq 1}
\Big(28  \Big( x^2_2 +x^2_4 + x^2_1\Big)+
28 \frac{\e^2}{z_1^2}||x||^4_{\e, |z_1|}  + 
28 \frac{\e^2}{z_1^2}x_3^2  +
48 \frac{\e}{z_1} ||x||^2_{\e, |z_1|}  x_1 \Big) \frac{|z_1| }{|x|_\e}
 dx.$$
Note that  $
||x||_{\e, |z_1|} \to   (x^2_1 + x^2_2 + x_4^2)^{1/2}$ as $\e \to 0$.
Hence letting $\e \to 0 $ and calling 
$||x||_{\H \times \mathbb{R}} = \Big((x^2_1 + x^2_2 + x_4^2)^2 + x_3^2\Big)^{1/4} $,
we get
$$I  \to - \frac{21}{C_{\H\times \R}} \int_{ ||x||_{\H\times \R}^4\leq 1}
  \Big( x^2_2 +x^2_4 + x^2_1\Big)dx=-1,$$
by the choice of $C_{\H\times \R}$, see \eqref{CHR}.
\end{proof}
\begin{proposition}\label{solfond}
The function $\Gamma$ satisfies the equation 
$$\Delta_z\Gamma(\cdot,z) =  -\delta_z + f_z, 
$$
where $\Delta_z$ is defined in \eqref{operator} and $f_z$ is defined in \eqref{fz}.

\end{proposition}

\begin{proof}    
Let us first note that if $u\in C_0^\infty$, then an integration by parts on the set $\{d(z,y) \geq\e\}$ ensures that 
$$\int_{d(z,y) \geq\e} \Delta_z \Gamma(y,z) u(y) |y| dy=$$
$$= \sum_{i=1}^3\int_{d(z,y) \geq\e}  Z^*_i   T_i \Gamma (y,z)u(y) dy  + 
\int_{d(z,y) \geq\e}  Z^*_{4, y-z}  T_{4, y-z} \Gamma(y,z)  u(y) dy=$$

$$=- \int_{d(z,y) =\e}\sum_{i=1}^3
 Z_i \Gamma(y,z)\nu_{Z_i}u(y)|y| d\sigma(y) 
-\int_{d(z,y) =\e}
Z_{4 ,y-z} \Gamma(y,z)\nu_{Z_{4 , y-z}} u(y)|y| d\sigma(y) 
-$$$$-\int_{d(z,y) \geq\e}\sum_{i=1}^3 Z_i \Gamma (y,z)Z_iu(y) |y| dy - \int_{d(z,y) \geq\e} Z_{4, y-z} \Gamma(y,z) Z_{4 , y-z}u(y) |y| dy,$$
where   $\nu$ is the external normal to the sphere. 
Applying Lemma \eqref{singular} we obtain the thesis. 
\end{proof}

\subsection{Rapresentation formula}

Let us note explicitly that Proposition 
\ref{solfond}
is indeed a representation of any $C_0^\infty$ function in terms of the vector fields $Z_i,$ for $i=1, 2, 3$ and $Z_{4, y-z}.$
The vector field 
 $Z_{4 ,y-z}$ can be applied to $\Gamma$, but  we would like to obtain a representation formulas where only horizontal derivatives are applied to $u.$ For this reason, we start by representing the formal adjoint of $Z_{4, y-z}$
in terms of the vector fields 
 $Z_i,$ for $i=1, 2, 3$  alone.
\begin{lemma}
\label{15ottobre}

Let $f,g \in C_0^\infty(M)$. Thus,
\begin{align}
\int & Z_{4, y-z} f g(y) dy 
    = \int \frac{y_3z_1}{|y|^2} f Z_1g(y) dy-\frac{3}{2}\int \frac{y_2z_1 }{|y|^2}f Z_2g(y) dy+\frac{3}{2}\int \frac{y_4z_1 }{|y|^2}    f Z_3g (y) dy-\nonumber
 \\&  -  \int 
\frac{y_3z_1}{|y|^3} fg(y) dy
  - \int \frac{<y,y-z>}{2|y|}   Z_3f(y)Z_2g(y)dy
+\int \frac{<y,y-z>}{2|y|} Z_2f(y)Z_3g (y)dy\nonumber
\end{align}
and
\begin{align}
\int &Z^*_{4, y-z} f g(y)dy  
    = -\int \frac{y_3z_1}{|y|^2} f Z_1g(y)dy+\frac{3}{2}\int \frac{y_2z_1 }{|y|^2}f Z_2g(y)dy-\frac{3}{2}\int \frac{y_4z_1 }{|y|^2}    f Z_3g (y)dy+ \nonumber
 \\& + 2 \int 
\frac{y_3z_1}{|y|^3} fg(y) dy
  + \int \frac{<y,y-z>}{2|y|}   Z_3f(y)Z_2g(y)dy
-\int \frac{<y,y-z>}{2|y|} Z_2f(y)Z_3g(y)dy.\nonumber
\end{align}

\end{lemma}
\begin{proof}
We first recall that  
$Z_{j} =  \sum_i \frac{(A_y)_{ij}}{|y|}\partial_{y_i}$, for $j=1,2,3,$
$|y|Z_{4} =  \sum_i \frac{(A_y)_{i4}}{|y|}\partial_{y_i}$, 
where $A$ is defined in \eqref{Ax}. 
It follows that  for every $i=1, \cdots,  4$
$$\partial_{y_i} = \sum_{j=1}^3\frac{(A_y)_{ij}}{|y|} Z_{j}+\frac{(A_y)_{i4}}{|y|} |y|Z_{4} =
\sum_{j=1}^3\frac{(A_y)_{ij}}{|y|} Z_{j}+\frac{(A_y)_{i4}}{2} [Z_2, Z_3].
$$ 
Plugging this expression in the equality $
T_{4, y-z}=\sum_i (A_{y-z})_{i4}\partial_{y_i}
$, we have
\begin{align*}
T_{4, y-z}=& 
\sum_{j=1}^3\sum_{i=1}^4 (A_{y-z})_{i4} \frac{(A_y)_{ij}}{|y|} Z_{j}+ 
\sum_{i=1}^4 (A_{y-z})_{i4}  (A_y)_{i4}\frac{[Z_2, Z_3]}{2}\\ =& - \frac{y_3 z_1Z_1 - y_2z_1 Z_2 + y_4z_1 Z_3}{|y|} +  
<y,y-z>\frac{[Z_2, Z_3]}{2},
\end{align*}
where we used the fact that $z=(z_1, 0,0,0).$ Then
\begin{align}\label{Z4perparti}\int Z_{4, y-z} f(y) g(y) dy  
=& -\int \frac{y_3z_1}{|y|^2}Z_1f (y)g(y) dy + \int \frac{y_2z_1 }{|y|^2}Z_2f (y)g(y) dy-\\ \nonumber &-\int \frac{y_4z_1 }{|y|^2}   Z_3 f(y)g(y) dy  + \int  \frac{<y,y-z>}{2|y|} [Z_2, Z_3]f(y)g(y) dy. \end{align}

Let us integrate by parts each term
in the right hand side of \eqref{Z4perparti}, starting from the first one. By Lemma \ref{aggiunti} we have
\begin{equation}\label{primopp}- \int \frac{y_3z_1}{|y|^2}Z_1f (y)g(y)dy = 
 \int  \frac{y_3z_1}{|y|^2}f Z_1g(y)dy +\int \frac{3y_3z_1}{|y|^3} f  g(y)dy +\end{equation}
$$+ \int Z_1\Big(\frac{y_3z_1}{|y|^2} \Big)  f g (y)dy
=\int \frac{y_3z_1}{|y|^2}  fZ_1g(y)dy +2\int  \frac{y_3z_1}{|y|^3}f  g (y)dy.$$
We integrate the second and third term in \eqref{Z4perparti}, using the fact   $
Z_2|y| = 
Z_3|y|=0$, as
proved in equation \eqref{Znorm}. Thus,
\begin{equation}\label{secondopp}
\int \frac{y_2z_1 }{|y|^2}Z_2f(y) g(y)dy-\int \frac{y_4z_1 }{|y|^2}   Z_3 f(y)g (y)dy
=\end{equation}
$$=-\int \frac{y_2z_1 }{|y|^2}f Z_2g(y)dy+\int \frac{y_4z_1 }{|y|^2}    f Z_3g(y)dy - 2 \int \frac{y_3z_1 }{|y|^3}    f g (y)dy
.$$
Now we integrate by parts the last  term
on the right hand side of \eqref{Z4perparti}. 
Since $z=(z_1, 0,0,0)$, then  $<y,y-z> = |y|^2 - y_1 z_1$. In addition  $Z_2 |y| = Z_3 |y|=0,$ hence, using the expression of the vector fields, we have:
$
-Z_2 \frac{y_2z_1}{|y|}=- \frac{y_3z_1}{|y|^2},$
$
Z_3 \frac{y_4z_1}{|y|}=- \frac{y_3z_1}{|y|^2}$. Therefore,
\begin{align}\label{ultimopp}\int  &\frac{<y,y-z>}{2|y|} [Z_2, Z_3]f(y)g(y)dy= 
 \\  \nonumber  =&-\int  \frac{<y,y-z>}{2|y|}   Z_3f(y) Z_2g(y)dy -\int  \frac{y_4z_1}{2|y|^2} Z_3f(y) g(y)dy +\\  \nonumber &+ \int \frac{<y,y-z>}{2|y|}   Z_2f (y)Z_3g(y)dy+ \int  \frac{y_2z_1}{2|y|^2}Z_2f(y) g(y)dy=
 \\  \nonumber  =&-\int  \frac{<y,y-z>}{2|y|}   Z_3f(y) Z_2g (y)dy+\int  \frac{y_4z_1}{2|y|^2} f (y)Z_3g (y)dy +\\  \nonumber &+ \int  \frac{<y,y-z>}{2|y|}   Z_2f(y) Z_3g(y)dy- \int  \frac{y_2z_1}{2|y|^2}f (y)Z_2g (y)dy-   \int 
\frac{y_3z_1}{|y|^3} fg(y) dy.
\end{align}
Inserting \eqref{primopp}, 
\eqref{secondopp}
and \eqref{ultimopp} in \eqref{Z4perparti} we obtain
\begin{align}
\int &Z_{4, y-z} f (y)g  (y)dy
    = \int \frac{y_3z_1}{|y|^2} f Z_1g(y)dy-\frac{3}{2}\int \frac{y_2z_1 }{|y|^2}f Z_2g(y)dy+\frac{3}{2}\int \frac{y_4z_1 }{|y|^2}    f Z_3g (y)dy-\nonumber
 \\& -   \int 
\frac{y_3z_1}{|y|^3} fg(y) dy
  - \int \frac{<y,y-z>}{2|y|}  (y) Z_3f(y)Z_2g
(y)dy+\int \frac{<y,y-z>}{2|y|} Z_2f(y)Z_3g(y)dy.\nonumber
\end{align}
Hence, by Lemma \ref{aggiunti}, 
\begin{align}
\int &Z^*_{4, y-z} f (y)g  
   (y)dy = -\int \frac{y_3z_1}{|y|^2} f Z_1g(y)dy+\frac{3}{2}\int \frac{y_2z_1 }{|y|^2}f Z_2g(y)dy-\frac{3}{2}\int \frac{y_4z_1 }{|y|^2}    f Z_3g (y)dy+\nonumber
 \\& +  2 \int 
\frac{y_3z_1}{|y|^3} fg(y) dy
  + \int \frac{<y,y-z>}{2|y|}   Z_3f(y)Z_2g
(y)dy-\int \frac{<y,y-z>}{2|y|} Z_2f(y)Z_3g(y)dy.\nonumber
\end{align}

\end{proof}

Let us start proving our main theorem in the special case of points $z=(z_1, 0,0,0)$.

\begin{proposition}
\label{quasigoal}
Let us fix $z=(z_1, 0,0,0)$ and consider $u\in C_0^\infty(M)$ 
\begin{align}  u(z)  =& \sum_{i=1}^3\int \Big(K_i(y,z) + Z_i \Gamma(y,z) \Big)Z_i u(y) |y| dy\nonumber\\&
 +\int \Big( K_0(y,z) + f_z(y)\Big) 
u(y) |y| dy 
\label{repformula},
\end{align}
where $f_z$ is defined in \eqref{fz} and the kernels are defined 
\begin{align}
&K_1 (y,z) :=  \frac{y_3z_1}{|y|^2} Z_{4 ,y-z} \Gamma(y,z)\label{ker},
\\& 
K_2 (y,z)
:= - 3\frac{y_2z_1 }{2|y|^2}Z_{4, y-z} \Gamma(y,z) - \frac{<y,y-z>}{2|y|}   Z_3Z_{4, y-z} \Gamma(y,z),\nonumber
\\& 
K_3 (y,z)
:=  3\frac{y_4z_1 }{2|y|^2}    Z_{4, y-z} \Gamma(y,z) + \frac{<y,y-z>}{2|y|} Z_2Z_{4, y-z} \Gamma(y,z),\nonumber
 \\& 
 K_0 (y,z)
:=-
 2 \frac{y_3z_1}{|y|^3}Z_{4, y-z} \Gamma(y,z).\nonumber
\end{align}

\end{proposition}

\begin{proof}
 Here we use the expression of the Laplace operator, and we integrate by parts, using the expression of $Z^*_{4, y-z}$ computed in Lemma \ref{15ottobre}. Then we  denote by $\nu$
the external horizontal normal to the set $\{d(z,y)\leq \e\}$,  hence when integrating by parts on the set $\{d(z,y) \geq\e\}$ we change sign. Thus, 
\begin{align}\nonumber
&     \int_{d(z,y) \geq\e} \Delta_z \Gamma(y,z) u(y) |y| dy\\ \nonumber
   =& \int_{d(z,y) \geq\e}  \sum_{i=1} ^3Z^*_i T_i\Gamma(y,z) u(y)   dy  
   + \int_{d(z,y) \geq\e}  \sum_{i=1} ^3Z^*_{4, y-z} T_{4 y-z} \Gamma(y,z) u(y)   dy\\ \nonumber
=& - u(z) \Bigg(\int_{d(z,y) =\e}\sum_{i=1}^3
 Z_i \Gamma(y,z) \nu_{Z_i} |y| d\sigma(y) 
+\int_{d(z,y) =\e}
Z_{4, y-z} \Gamma(y,z)\nu_{Z_{4, y-z}} |y| d\sigma(y) \Bigg) - 
\\\nonumber&
- \int_{d(z,y) =\e}\Bigg(\sum_{i=1}^3
  Z_i \Gamma(y,z)\nu_{Z_i} 
+
Z_{4, y-z} \Gamma(y,z) \nu_{Z_{4, y-z}} \Bigg) \Big(u(y) - u(z) \Big)|y| d\sigma(y)  -
\\ \nonumber&
-\int_{d(z,y) \geq\e}\sum_{i=1}^3 Z_i \Gamma (y,z)Z_iu |y| dy 
 - \int_{d(z,y) \geq\e} \frac{y_3z_1}{|y|^2} T_{4, y-z} \Gamma(y,z) Z_1u(y) dy+\\& +\frac{3}{2}\int_{d(z,y) \geq\e} \frac{y_2z_1 }{|y|^2}T_{4 ,y-z} \Gamma(y,z) Z_2u(y) dy-\frac{3}{2}\int_{d(z,y) \geq\e} \frac{y_4z_1 }{|y|^2}    T_{4, y-z} \Gamma(y,z) Z_3u(y) dy +\nonumber
 \\&+2
\int_{d(z,y) \geq\e}  \frac{y_3z_1}{|y|^3}T_{4, y-z} \Gamma(y,z) u(y) dy
  + \int_{d(z,y) \geq\e} \frac{<y,y-z>}{2|y|}   Z_3T_{4, y-z} \Gamma(y,z)Z_2u
(y) dy- \nonumber\\& -\int_{d(z,y) \geq\e} \frac{<y,y-z>}{2|y|} Z_2T_{4, y-z} \Gamma(y,z)Z_3u(y) dy.\nonumber
\end{align}
Using the fact that $\Delta_z \Gamma(y,z) = f_z$ out of the pole, 
and  the definition of the kernels $K_i,$ for $i=0, \cdots,  3$, given in \eqref{ker}, 
we get
\begin{align}
\label{calcoloDelta}    \int_{d(z,y) \geq\e} &f_z(y) u(y) |y| dy=\\
= - &u(z) \Bigg(\int_{d(z,y) =\e}\sum_{i=1}^3
 Z_i \Gamma(y,z) \nu_{Z_i} |y| d\sigma(y) 
+\int_{d(z,y) =\e}
Z_{4, y-z} \Gamma(y,z)\nu_{Z_{4, y-z}} |y| d\sigma(y) \Bigg) - \nonumber
\\
\nonumber
-& \int_{d(z,y) =\e}\Bigg(\sum_{i=1}^3
  Z_i \Gamma(y,z)\nu_{Z_i} 
+
Z_{4 ,y-z} \Gamma(y,z) \nu_{Z_{4, y-z}} \Bigg) \Big(u(y) - u(z) \Big)|y| d\sigma(y)  -
\\\nonumber
-&\int_{d(z,y) \geq\e}\sum_{i=1}^3 Z_i \Gamma (y,z)Z_iu |y| dy - 
\sum_{i=1}^3\int_{d(z,y) \ge \e} K_i (y,z) Z_i u(y) |y| dy-\nonumber \\ \nonumber 
 -& \int_{d(z,y) \ge \e}  K_0(y,z)  
u(y) |y| dy.
\end{align}

Let us consider the first term in the right hand side of \eqref{calcoloDelta}, and denote it by $I_{1, \e}$. 
Letting  $\e$ to $0$  
and applying Lemma \ref{singular} we obtain 
$$I_{1, \e}  \to  u(z).$$
We denote $I_{2, \e}$ the second term in the right hand side of \eqref{calcoloDelta}. We note that $|u(z)-u(y)| \leq |\nabla_E u|_{\infty}|z-y|\leq |\nabla_E u|_{\infty} d(y,z)= |\nabla_E u|_{\infty}\e,$
and that 
$Z_i \Gamma(y,z) \nu_{Z_i} = 
\frac{|Z_i \Gamma(y,z)|^2}{|\nabla_E\Gamma(y,z)|}\geq 0, 
$
where $\nabla_E$ is the Euclidean gradient.
As a result, 
$$|I_{2, \e}|\leq |\nabla_Eu|_{\infty} \e \int_{d(z,y) =\e}\Bigg(\sum_{i=1}^3
  Z_i \Gamma(y,z)\nu_{Z_i} 
+
Z_{4 ,y-z} \Gamma(y,z) \nu_{Z_{4 y-z}} \Bigg) |y| d\sigma(y) = C \e I_{1, \e} \to 0 $$
as $\e\to 0$. 
From here the thesis follows immediately.  
\end{proof}

Note that the operators $K_i$  are of order $1$ according to Definition \ref{locdeg}. Moreover they
satisfy the following estimates.

\begin{lemma}\label{stimecappa}
With the notation of the previous proposition,
there exists a constant $C>0$ such that 
$$|K_0(y,z)|\leq C d^{-Q+1}(y,z)\frac{|z|^2}{|y|^3}, \quad |K_i(y,z)|\leq C d^{-Q+1}(y,z)\frac{|z|^2}{|y|^2},$$
for every $i=1,2, 3$. 

\end{lemma}
\begin{proof}
 Using the expression of $Z_{4, y-z} \Gamma(y,z)$ contained in  \eqref{venerdi},  we immediately get
$$|K_1 (y,z)| =\frac{|y_3z_1|}{|y|^2}
|Z_{4, y-z} \Gamma(y,z)| = 
6 d^{-7}(y,z)   |(y-z)_1| \frac{y_3^2 |z_1|^3}{C_{\H\times \R}|y|^3}\leq C d^{-Q+1} (y,z)\frac{|z|^2}{|y|^2}.$$
In order to compute $K_2$, we differentiate  $Z_{4, y-z} \Gamma(y,z)$ and, using    that $Z_3|y|=0,$ we get
\begin{align*}
Z_3Z_{4, y-z} \Gamma(y,z) = & 
\frac{6}{|y|} Z_3\Big( d^{-7}(y,z)   (y-z)_1 y_3 z_1^2\Big)\\ =& \frac{21}{2|y|}d^{-11}(y,z)  Z_3d^4  (y-z)_1 y_3 z_1^2 +\\ & +\frac{6}{|y|}d^{-7}(y,z)  Z_3  (y-z)_1 y_3 z_1^2 +\frac{6}{|y|}d^{-7}(y,z)  (y-z)_1 Z_3y_3 z_1^2 \\ =& \frac{21}{2}d^{-11}(y,z)  \frac{1}{|y|^2}\Big( 4 |y-z|^2 z_1 y_2 + 8 y_4 y_3 z_1^2\Big) (y-z)_1 y_3 z_1^2 -\\ & -6d^{-7}(y,z) \frac{y_2}{|y|^2}y_3 z_1^2 +6d^{-7}(y,z)  (y-z)_1 \frac{y_4}{|y|^2} z_1^2 .\end{align*}
Hence 
$$
|Z_3Z_{4, y-z} \Gamma(y,z)|
\leq C d^{-Q}(y,z) \frac{z_1^2}{|y|^2}.$$
From here we deduce the estimate
$$|K_2 (y,z) |\leq C d^{-Q+1} (y,z)\frac{z_1^2}{|y|^2}.$$
Similarly 
$$|K_3 (y,z) |\leq C d^{-Q+1} (y,z)\frac{z_1^2}{|y|^2}\,.$$
Finally 
$$
|K_0 (y,z)| = \frac{C}{|y|}|K_1 (y,z)|
\leq C d^{-Q+1} (y,z)\frac{z_1^2}{|y|^3}.
$$
\end{proof}

\begin{remark}
Let us explicitly note that, if $z$ is fixed and different from $0$,  the kernels $K_i$, with $i=0, \cdots, 3$  have two singularities, one in $z,$ and one in the point $0.$ This is a new phenomena, 
since the Euclidean analogous of formula \eqref{repformula} only contains the gradient of $\Gamma,$ which only has a singularity in $z$. However, both singularities are integrable,   the behavior of the integral representation is the exact analogous of the  Euclidean one. The same argument also applies to the function $f_z,$ defined in \eqref{fz}  which turns out to have a singularity in $0,$ thanks to the estimate \eqref{DeltaGamma}.

\end{remark}

In the following lemma we consider what happens where $z \to 0$: the two singularities of the kernels $K_i$ do not lead to a higher order singularity.

\begin{lemma}
Let $u\in C_0^\infty(M)$.  Then, for $z\to 0,$ $$\int K_0(y,z) u(y) |y| dy\to 0, \int f_z(y) u(y) |y| dy\to 0 $$
 and 
\begin{align} &\sum_{i=1}^3\int K_i(y,z) Z_i u(y) |y| dy\nonumber\to \\&
- \int \frac{|y|^2}{2}   Z_3Z_{4} \Gamma(y,0)Z_2 u(y) |y|dy+ \int\frac{|y|^2}{2} Z_2Z_4 \Gamma(y,0)Z_3 u(y) |y|dy
\label{ohno},
\end{align}
as $z\to 0$. 

\end{lemma}
\begin{proof}
Indeed, 
if $z\to 0,$ we can assume that $z\not= 0,$ and choose $\e = \sqrt{|z|}$. Hence  we can write 
$$\int K_0(y,z) u(y) |y| dy = 
\lim_{\e \to 0}
\int_{d(z,y)\geq  \e} K_0(z,y) u(y) |y| dy.$$
Since $d(y,z)\geq \e,$ then
$|y| = d(y,0) \geq d(y,z) - d(z,0) \geq  \e- \e^2 = \e(1 - \e) \geq \frac{1}{2}\e$ so that 
$$\int K_0(y,z) u(y) |y| dy\leq 
\e \sup|u|\int_{B_M(0,r)} d^{-Q+1}(z,y)  |y| dy \leq c\e r,$$
where $B_M(0,r)$ is a ball containing  the  support of $u$, by  Lemma \ref{misurasfera} and \eqref{dQ1}. Letting  $\e\to 0$ we immediately obtain the first assertion of the thesis. The same argument also applies to the  integral containing  $f_z$, by choosing $\e=|z|^{1/4}$.
In addition, we get as $z\to 0$
\begin{align*} &\sum_{i=1}^3\int K_i(y,z) Z_i u(y) |y| dy\nonumber\to \\&
- \int \frac{|y|}{2}   Z_3Z_{4, y}  \Gamma(y,0)Z_2 u(y) |y|dy+ \int\frac{|y|}{2} Z_2Z_{4, y} \Gamma(y,0)Z_3 u(y) |y|dy.
\end{align*}
\end{proof}

\begin{remark} 
If $z=0$ and 
$u\in C_0^\infty(M)$, using the expression of the commutator of $Z_2,$ $Z_3$ we obtain
\begin{align}
- &\int \frac{|y|}{2}   Z_3Z_{4, y} \Gamma(y,0)Z_2 u(y) |y|dy+ \int\frac{|y|}{2} Z_2Z_{4, y} \Gamma(y,0)Z_3 u(y) |y|dy= 
\nonumber\\&\nonumber
= \int \frac{|y|}{2}   Z_{4, y} \Gamma(y,0)[Z_3,Z_2] u(y) |y|dy=\int    Z_{4} \Gamma(y,0)Z_{4, y}u(y) |y|dy.\end{align}
Using the fact that  $(Z_i)_{i=1, 2, 3}, Z_{4, y}$ are an orthonormal basis we get
\begin{align}\sum_{i=1}^3\int \Big(K_i(y,0) + Z_i \Gamma(y,0) \Big)Z_i u(y) |y| dy\nonumber&\to 
\frac{3}{2}\int \langle\nabla_E  \big(|y|\Gamma(y,0) \big), \nabla_E u(y) \rangle dy=\\
& = \frac{3}{2}\int \langle\nabla_E   \Gamma_E(y,0), \nabla_E u(y) \rangle dy.
\end{align}
This expression is compatible with the fact that $|y|\Gamma(\cdot, 0) = \Gamma_E(y,0) $ is the Euclidean fundamental solution at the origin. 
\end{remark}

Up to now, we obtained a representation formula  for any smooth function $u$  at  points  $z=(z_1,0,0,0)$ , $z_1\ge 0$. Using the invariance properties of $\Gamma,$ we can prove in the sequel that an analogous of Proposition \ref{quasigoal} holds at every point $z$.

\begin{remark}\label{chvarxi}
Let $g\in C_{0}^\infty(M)$
,  let $x \in M$ such that $|x|=1$ and let us  set 
\begin{equation}\label{questanotazionequi} g_{x}(y): =  g(A_{x}y).\end{equation}

Since $A_{x}n=x$, then  $g_{x} (n) = g(x)$. Moreover, 
since the vector fields $T_i$, for $i=1, 2,3$ are invariant with respect to the rotation on the $M$ plane, then 
$$( T_i g ) (A_{x}y) =  T_i g_
{x}(y), \ \ \forall y \in M.$$
In particular, for every $x$ such that $|x|=1$, we have
$$( T_i g ) (x)=( T_i g ) (A_{x} n) =  T_i g_
{x}(n)$$
and 
$$T_i \Gamma(A_{x}(y), A_{x}( z)) = T_i \Gamma(y, z)\ \ \forall y, z \in M,$$
 $i=1, 2, 3$.

\end{remark}

We need to extend on the whole space the function $f_z$  and the kernels  $K_i$ $i=0, \cdots , 3$, previously defined  only for $z=(z_1,0,0,0)$,  $z_1\geq 0$ (see formulas \eqref{fz} and  \eqref{ker}). For a general value of $z\not =0$, we impose rotational invariance and we set  
\begin{equation}\label{extend}
 f_z (y) := f_{|z|n}\Big( A^{-1}_{\frac{z}{|z|}}( y)\Big) \quad K_i (y, z) := K_i \Big(A^{-1}_{\frac{z}{|z|}}(y), |z| n\Big), 
\end{equation}
for all $i=0, \cdots, 3$.

We can now show that the representation formula stated in  Proposition \ref{quasigoal} holds for every point in $M$.

\begin{theorem}\label{chvar}
Consider $u\in C_0^\infty(M)$.  Then
\begin{align*}  u(z) & =
\sum_{i=1}^3\int \Big(K_i(y,z) + Z_i \Gamma(y,z) \Big)Z_i u(y) |y| dy + 
 \int \Big( K_0(y,z) + f_z(y)\Big) 
u(y) |y| dy,
\end{align*}
where the kernels $K_i$ have been defined in \eqref{extend}. 
\end{theorem}

\begin{proof}
Let us consider a point $x \in M$, 
and let us call 
$$P(u)(z) =  \sum_{i=1}^3\int \Big(K_i(y,z) + Z_i \Gamma(y,z) \Big)Z_i u(y) |y| dy+ 
\int \Big( K_0(y,z) + f_z(y)\Big) 
u(y) |y| dy . $$
Let us start by proving that the operator $P$ satisfies
$$P(u)(x) = P(u_{\frac{x}{|x|}})(|x|n).$$

By \eqref{19 maggio}, we have that 
$|z|\,n = A_{\frac{ z}{|z|}}^{-1}( z)$. 
For the invariance properties of $\Gamma$ 
stated in Remark \ref{invariancegamma}, we have 
$$Z_i\Gamma(y,z) =  Z_i \Gamma \big(  A_{ \frac{ z}{|z|}}^{-1}( y) , A_{ \frac{ z}{|z|}}^{-1}( z)\big)$$
and, from the definition of $K_i$,  in \eqref{extend} we have
$K_i(y,z) = K_i \Big( A_{ \frac{ z}{|z|}}^{-1}( y),   |z|n\Big)$, and the same  holds for $f_z$. 
Hence, we get
\begin{align*}
P (u)(z) 
=& \sum_{i=1}^3\int K_i \Big(  A_{ \frac{ z}{|z|}}^{-1}( y) , |z|n\Big)Z_i u(y) |y| dy
+ 
\sum_{i=1}^3\int
Z_i \Gamma \Big(  A_{ \frac{ z}{|z|}}^{-1}( y) , A_{ \frac{ z}{|z|}}^{-1}( z)\Big)
Z_i u(y) |y| dy
+ \\&
+\int  K_0\Big( A_{ \frac{ z}{|z|}}^{-1}( y),   |z|n\Big) 
u(y) |y| dy + \int  f_{|z|n}\Big( A_{ \frac{ z}{|z|}}^{-1}( y)\Big) 
u(y) |y| dy .\end{align*}
With the change of variable  $x = A_{\frac{ z}{|z|}} ^{-1} (y)$ we  write
\begin{align*}
P(u)(z)  =&  
\sum_{i=1}^3\int \Big(K_i (x, |z| n)+ Z_i \Gamma (x, |z| n)\Big)Z_i u(A_{ \frac{ z}{|z|}}( x)) |x| dx + \\
&+
\int  K_0(x,  |z|n ) 
u(A_{ \frac{ z}{|z|}}( x)) |x| dx + \int  f_{|z|n}(x) 
u(A_{ \frac{ z}{|z|}}( x)) |x| dx  =
P (u_{\frac{ z}{|z|}})(|z|n),
\end{align*}
where 
$u_{\frac{z}{|z|}}$ has been defined in \eqref{questanotazionequi}. 
Applying the last equality and Proposition \ref{quasigoal} we finally obtain $$ P(u)(z)=
P (u_{\frac{ z}{|z|}})(|z|n) = 
u_{\frac{ z}{|z|}}(|z|n)= 
u(z), $$
where the last equality follows once again by the definition of $u_{\frac{ z}{|z|}} $ and keeping in mind that $|z|\,n = A_{\frac{ z}{|z|}}^{-1}( z)$.
This concludes the proof.
\end{proof}

\end{document}